\documentclass{amsart}
\usepackage{enumerate}
\usepackage{graphs}
\usepackage{graphicx}

\newcommand{\specabs}{\operatorname{spec}_{\operatorname{abs}}}

\newcommand{\Eig}{{\rm Eig}}
\newcommand{\Piperp}{\Pi^\perp}
\newcommand{\lmin}{l_{\operatorname{min}}}

\renewcommand{\epsilon}{\varepsilon}
\renewcommand{\phi}{\varphi}
\newcommand{\Xinfty}{X^\infty}
\newcommand{\XNG}{X^N_G}
\newcommand{\DeltaXinfty}{\Delta_{X^\infty}}
\newcommand{\DeltaXNG}{\Delta_{X_G^N}}

\newcommand{\alphamax}{\alpha_{\rm{max}}}

\newcommand{\tmax}{t_{\rm{max}}}
\newcommand{\Bperp}{B_\perp}
\newcommand{\Etilde}{\tilde{E}}
\newcommand{\distsymm}{\dist_{\rm{symm}}}
\newcommand{\Eigapp}{\Eig^{{\rm app}}}
\newcommand{\lambdamax}{\lambda_{\rm max}}
\newcommand{\ubar}{\overline{u}}
\newcommand{\la}{\langle}
\newcommand{\ra}{\rangle}

\begin{document}
\title{Spectra of graph neighborhoods and scattering}

\author{Daniel Grieser}
\address{Institut f\"ur Mathematik, Carl von Ossietzky Universit\"at Oldenburg, D-26111 Oldenburg}
\email{grieser@mathematik.uni-oldenburg.de}
%\datver{October 9, 2007}
\begin{abstract}
Let $(G_\epsilon)_{\epsilon>0}$ be a family of '$\epsilon$-thin' Riemannian manifolds modeled on a finite metric graph $G$, for example, the $\epsilon$-neighborhood of an embedding of $G$ in some Euclidean space with straight edges. We study the asymptotic behavior of the spectrum of the Laplace-Beltrami operator on $G_\epsilon$ as $\epsilon\to 0$, for various boundary conditions. We obtain complete asymptotic expansions for the $k$th eigenvalue and the eigenfunctions, uniformly for $k\leq C\epsilon^{-1}$,  in terms of scattering data on a non-compact limit space. We then use this to determine the quantum graph which is to be regarded as the limit object, in a spectral sense, of the family $(G_\epsilon)$.

Our method is a direct construction of approximate eigenfunctions from the scattering and graph data, and use of a priori estimates to show that all eigenfunctions are obtained in this way.
\vspace{-1cm}
\end{abstract}
\maketitle

\tableofcontents

\section{Introduction}
Consider a graph $G$ with a finite number of vertices and edges embedded in $\R^n$ with straight edges. For $\epsilon>0$ let $G_\epsilon$ be the set of points of distance at most $\epsilon$ from $G$. For small $\epsilon$, $G_\epsilon$ 'looks almost like' $G$; in other words, the one-dimensional object $G$ should be considered as
a good model for $G_\epsilon$, and so one expects many physical and analytical properties of $G_\epsilon$ to be understandable, in an approximate sense, by corresponding properties of $G$. The property we analyze in this light here is the spectrum of the Laplacian, under various boundary conditions. For motivation from physics, see for example \cite{Kuc:GMWTS}.

This problem has received some attention in the last decade. The case of Neumann boundary conditions was analyzed at various levels of generality in \cite{Col:MPVPNNL}, \cite{FreWen:DPGAP}, \cite{KucZen:CSMSCG}, \cite{RubSch:VPMCTSI}, \cite{ExnPos:CSGLTM}, where it was shown that for each $k\in\N$ the $k$th eigenvalue of the Neumann Laplacian on $G_\epsilon$ converges, as $\epsilon\to 0$, to the $k$th eigenvalue of the second derivative operator on the union of the edges of $G$, where at the vertices so-called Kirchhoff boundary conditions are imposed. The question of what the corresponding limiting behavior is for other, for example Dirichlet, boundary conditions was characterized as 'very difficult' in \cite{ExnPos:CSGLTM} and remained open until some partial progress was made recently, see Section \ref{subsec previous}.

In this paper we solve this problem for a general mixed boundary problem, where we impose Dirichlet boundary conditions on one part of the boundary and Neumann conditions on the rest. Other boundary conditions, for example Robin, are also possible. Instead of the setting of an embedded graph described above, we consider the more general, and mathematically more natural, situation of a shrinking family of Riemannian manifolds modeled on the graph: We now consider $G$ as an abstract metric graph, that is,  for each edge $e$  a positive number $l_e$ (to be thought of as half the edge length) is given. In addition, we have geometric data: For each edge $e$
an $(n-1)$-dimensional Riemannian manifold $Y_e$, and for each edge $v$ an $n$-dimensional Riemannian manifold $X_v$; all these manifolds are compact and may have a piecewise smooth boundary; also, for each edge $e$ incident to a vertex $v$ (denoted $e\sim v$), we are given an identification (gluing map) of $Y_e$ with a subset of the boundary of $X_v$, without overlaps. Then $G_\epsilon$ is defined by gluing cylinders of length $2l_e$ and cross section $\epsilon Y_e$ (the factor denotes a rescaling of the metric) to
$\epsilon X_v$ for all pairs $e\sim v$. Also given are subsets D and N of the boundaries of each $X_v$ and $Y_e$, which yield corresponding subsets of $G_\epsilon$ (or, for Robin conditions, similar data). We investigate the Laplace-Beltrami operator on $G_\epsilon$, where Dirichlet boundary conditions are imposed on D and Neumann conditions on N. We suppose sufficient regularity of the boundary and the D/N decomposition so that this problem has discrete spectrum.

Besides the $\epsilon$-neighborhoods mentioned before this manifold setting includes the case of the boundary of the $\epsilon$-neighborhood of $G$ (in this case, $G_\epsilon$ itself has no boundary).
Our results imply the following theorem.
\begin{theorem}
\label{maintheorem_Gepsilon}
Let $\mu_1(\epsilon)\leq \mu_2(\epsilon)\leq\dotsc$ be the eigenvalues of the Laplacian on $G_\epsilon$, with the given boundary conditions, repeated according to their multiplicity.

For each edge $e$, let $\nu_e$ be the smallest eigenvalue of $-\Delta$ on $Y_e$, with the given boundary conditions, and let $\nu=\min_e \nu_e$.

There are numbers $D\in\N_0$, $0<\tau_1\leq \tau_2\leq\dotsc\leq \tau_D\leq \nu$
and $0\leq b_1\leq b_2\leq\dots\to\infty$ so that for $\epsilon\to 0$
\begin{align}
\label{eqnlowev Gepsilon}
\mu_k(\epsilon) &= \epsilon^{-2} \tau_k + O(e^{-c/\epsilon}),&& k=1,\dots,D\\
\label{eqnhighev Gepsilon}
\mu_k(\epsilon) &= \epsilon^{-2} \nu + b_{k-D} + O(\epsilon),&& k>D
\end{align}
\end{theorem}

In fact, we obtain much more precise information: complete asymptotics  up to errors of the form $O(e^{-c/\epsilon})$ also in the case $k>D$, uniform estimates not only for fixed $k$ but also for $k$ on the order of $\epsilon^{-1}$ as well as detailed information on the eigenfunctions, see the theorems below.

For Theorem \ref{maintheorem_Gepsilon} to be meaningful it is essential to identify the numbers $\tau_k$ and $b_k$ in terms of the given data. We do this in two steps: First, we give general formulas identifying these numbers in terms of scattering data of an associated non-compact $n$-dimensional limiting problem. Second, we analyze  how this scattering data can be obtained from the given data. It turns out that in this second step the  situation of general boundary conditions carries some essential new features compared to Neumann boundary conditions: while in this case one has $D=0$ and the $b_k$ are determined by the graph and the edge lengths alone (that is, independent of the manifolds $X_v$, $Y_e$), in general $D>0$ and the $\tau_k$ and $b_k$ depend also on transcendental analytic information involving the manifolds $X_v$ and $Y_e$: The $\tau_k$ are $L^2$-eigenvalues in the afore-mentioned limit problem, and the $b_k$ are eigenvalues of a quantum graph whose vertex boundary conditions involve the scattering matrix of that limit problem, see Theorems \ref{maintheorem1} and \ref{mainthmquantumgraph}.

%To motivate our main results, we look at the simple example of a $1\times\epsilon$ %rectangle with Dirichlet boundary conditions. The $k$th eigenvalue is %$\epsilon^{-2}\pi^2+k^2\pi^2$, for $\epsilon<\frac1k$. The diverging term %$\epsilon^{-2}\pi^2$ arises from the short cross sectional interval $[0,\epsilon]$, %and its coefficient $\pi^2$ is the first Dirichlet eigenvalue on this interval. The %subleading term is the $k$th eigenvalue on the 'long' interval $[0.1]$. Our main %result identifies the corresponding coefficients in the general case.

\subsection{Main results}
We now state our results more precisely.
Because of  the divergence in \eqref{eqnlowev Gepsilon},  \eqref{eqnhighev Gepsilon} it is more convenient, and for a proper understanding it turns out to be essential, to rescale the problem: We multiply all lengths by $\epsilon^{-1}$. We denote
$$N:=\epsilon^{-1},\quad X^N_G := \epsilon^{-1} G_\epsilon$$
This rescales the eigenvalues by the factor $\epsilon^2$. In $X^N_G$, the vertex manifolds and the cross sections of the edges are independent of $N$ while the lengths of the edges are $2Nl_e$, and we are interested in the limit $N\to\infty$. Central to the analysis is the limit object, to be thought of as $\lim_{N\to\infty} X^N_G$,
\begin{equation}
\label{eqn def Xinfty}
X^\infty := \bigcup_v X_v^\infty \quad \text{ (disjoint union)}
\end{equation}
where the 'star' $\Xinfty_v$ of a vertex $v$ of $G$ is obtained by attaching a half infinite cylinder  $[0,\infty)\times Y_e$ to the vertex manifold $X_v$, for each edge $e$ incident to $v$, see Figure \ref{figure1}.
Denote by $Y$ the cross-section for $X^\infty$, i.e.\ the disjoint union of the $Y_e$, each one appearing twice (once for each endpoint). The graph structure is encoded by a map $\sigma:Y\to Y$ which toggles the two copies of $Y_e$ for each edge $e$.
See Section \ref{sec setup} for  precise definitions.
\begin{figure}
\includegraphics[scale=0.8]{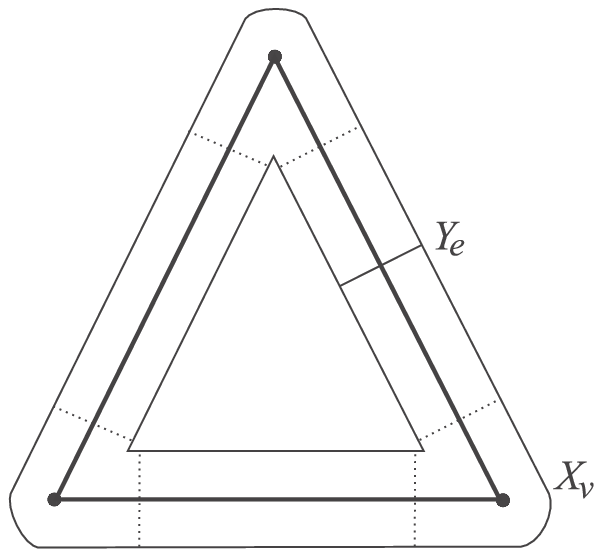} \hfill \includegraphics[scale=0.6]{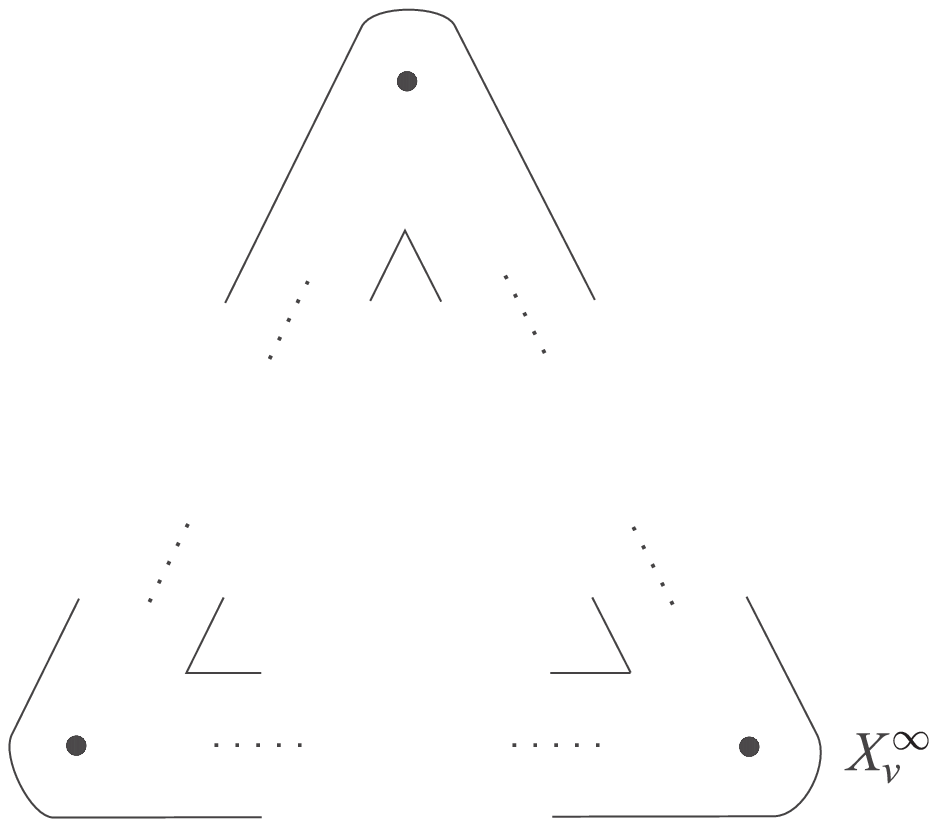}
\caption{The setup: $X^N_G$ and $\Xinfty$}
\label{figure1}
\end{figure}

Since $X^\infty$ is a non-compact space, the spectrum of the Laplacian $-\Delta_{X^\infty}$ (with corresponding boundary conditions) is no longer discrete. Its spectral theory is obtained via scattering theory, that is, one considers it as compact perturbation of $[0,\infty)\times Y$, the disjoint union of the edge cylinders. Let $\nu<\nu_1$ be the smallest and second smallest eigenvalue of the Laplacian of $Y$ and let $V$ be the eigenspace corresponding to $\nu$. The absolutely continuous spectrum of $-\Delta_{X^\infty}$ is $[\nu,\infty)$. To each $\lambda=\nu+\alpha^2\in[\nu,\nu_1)$ is associated  the scattering matrix $S(\alpha)$, a linear map $V\to V$ depending holomorphically on $\alpha$, and to each $\phi\in V$ a scattering solution (generalized eigenfunction) $E_{\alpha,\phi}$. These are the main objects of scattering theory.  In addition, $-\DeltaXinfty$ may have discrete spectrum ($L^2$-eigenvalues).

We also have linear maps $L,\sigma:V\to V$ where $L$ encodes the edge lengths (it equals $l_e\Id$ on the subspace of $V$ corresponding to the edge $e$) and $\sigma$ is induced by the graph structure map above. The maps $\sigma$ and $S(0)$ are involutions on $V$, so they have only the eigenvalues $\pm 1$.

We need the following data derived from the spectral data of $-\Delta_{X^\infty}$:

\begin{itemize}
\item
The $L^2$-eigenvalues $\tau_1'<\tau_2'<\dotsc < \tau_d'<\nu_1$ of $-\DeltaXinfty$, with eigenspaces $\calE_1,\dotsc,\calE_d\subset L^2(X^\infty)$;
\item
the space $\calP_0 := \Ker (\Id-\sigma) \cap \Ker (\Id-S(0))\subset V$.
\item
the spaces $\calP_z := \Ker (\Id - e^{i z2 L}\sigma S(0))\subset V$ for $z>0$. This is non-zero for a discrete set
\begin{equation}\label{eqn ai def}
0<z_1<z_2<\dotsc\quad\text{the zeroes of } f(z) = \det (\Id - e^{i z 2 L}\sigma S(0))
\end{equation}
\item
for each $i\in\N$, pairwise different holomorphic functions $Z_i^\rho(\alpha)$, $\rho=1,\dots,r_i$ with $r_i\in\N$, defined on $\C\setminus\{\alpha\in\R:\,|\alpha|\geq\nu_1\}$ and with
\begin{equation}
\label{eqn Z initial}
Z_i^\rho (0) = z_i
\end{equation}
for each $\rho$, and holomorphically varying subspaces $\calP_{i}^\rho\subset V$ with $\bigoplus_{\rho=1}^{r_i} \calP_{i}^\rho (0)=\calP_{z_i}$.
\end{itemize}
The $Z_i^\rho$ and $\calP_{i}^\rho$ arise from a bifurcation analysis of $\Ker (\Id - e^{i z 2 L}\sigma S(\alpha))$, see Theorem \ref{thmscattef1}.
\begin{theorem}[{\bf Main Theorem}]
\label{maintheorem1}
Denote by $-\DeltaXNG$ the Laplace-Beltrami operator on $\XNG$, with the given boundary conditions. For any $\lambdamax<\nu_1$ there are $c,N_0>0$ such that for $N>N_0$  the spectrum of $-\DeltaXNG$ in $[0,\lambdamax]$  consists of
(counting multiplicities)
\begin{enuma}
\item
$\dim \calE_p$ many eigenvalues of the form $\tau_p' + O(e^{-cN})$, for $p=1,\dotsc,d$,
\item
$\dim \calP_0$ many eigenvalues of the form $\nu + O(e^{-cN})$,
\item
$\dim \calP_{i}^\rho$ many eigenvalues of the form
\begin{equation}
\label{eqn alpha-formel}
\nu+\beta^2 \text{ with }\beta = \frac1{N} Z_i^\rho (\frac1{N}) + O(e^{-cN})
\end{equation}
 for each $i\in\N$  and $\rho=1,\dots,r_i$ for which $\nu+\beta^2\leq\lambdamax$.
\end{enuma}
The constants $c,N_0$ and the constants in the big $O$ only depend on the spectral data of $\DeltaXinfty$.
In particular, the asymptotics are uniform for eigenvalues $\leq\lambdamax$.
\end{theorem}
In terms of $\epsilon=N^{-1}$ \eqref{eqn alpha-formel} gives an eigenvalue $\mu=\nu+\epsilon^2z_i^2 + O(\epsilon^3)$ by \eqref{eqn Z initial}, so this implies
Theorem \ref{maintheorem_Gepsilon}: take $D= \sum_{p=1}^{d'} \dim \calE_p$ with $d'=\max\{p:\tau_p'\leq \nu\}$; the $\tau_k$ are the $\tau_p'$ repeated $\dim \calE_p$ times; and the $b_k$ are $0$ repeated $\dim\calP_0$ times, and the numbers $z_i^2$ repeated  $\dim \calP_i$ many times. The $\tau_p$ with $p>d'$ are invisible in Theorem \ref{maintheorem_Gepsilon} since there $k$ is fixed as $\epsilon\to 0$.

Theorem \ref{maintheorem1} also describes eigenvalue $\mu_k$ with $k\to\infty$ as $\epsilon\to 0$. We determine the permitted range of $k$. From $\mu_k=\nu+\epsilon^2 z_i^2 + O(\epsilon^3)\leq\lambdamax$ one obtains $\epsilon z_i<\sqrt{\lambdamax-\nu}+O(\epsilon)$, and the Weyl asymptotics for the $z_i$, see Proposition \ref{prop alpha=0}, show that this is equivalent to
\begin{equation}
\label{eqn k epsilon bound}
k\leq \epsilon^{-1} \frac{\sum_e l_e}\pi \sqrt{\lambdamax-\nu} + O(1).
\end{equation}
See Theorem \ref{thm alphaleadingterm} for a more explicit description of eigenvalues with $\beta$ close to some $\beta_0>0$.

We can also describe the eigenfunctions in terms of the scattering solutions and of the eigenfunctions in $\calE_p$, with exponentially small errors, see Theorems \ref{thm ev < nu}, \ref{thm main general theorem lambda>nu} and \ref{thm general theorem lambda close to nu} and Corollary \ref{cor bounded i} for precise statements.

The restriction $\lambdamax<\nu_1$ is made to keep the exposition at a reasonable length. Our methods are such that a generalization to higher eigenvalues, taking into account several thresholds, should be fairly straightforward.

To identify the numbers $b_k$ in Theorem \ref{maintheorem_Gepsilon} in terms of the data, we recall that a quantum graph is given by a metric graph ($G$ with the edge lengths $2l_e$) together with a self-adjoint realization of the operator $-d^2/d\xi^2$ acting on functions defined on the disjoint union of the edges, where $\xi$ is a variable along each edge, measuring length; such a self-adjoint extension is given by boundary conditions at the vertices. We state the boundary conditions obtained in our setting. Denote by
$$ \calN_e := \ker(\Delta_{Y_e}+\nu)$$
the eigenspace for $\nu$ of the Laplacian on the cross section for an edge $e$.
For each vertex $v$ let $V_v=\bigoplus_{e\sim v}\calN_e$, then clearly $V=\bigoplus_v V_v$ corresponding to the decomposition \eqref{eqn def Xinfty}. Also, the scattering matrix $S(\alpha)$ is the direct sum
of scattering matrices $S_v(\alpha):V_v\to V_v$. For each vertex $v$, $S_v(0)$ is an involution on $V_v$. Denote by $u_e$ the restriction of a function on $G$ to the edge $e$, and by $\partial_n u_e(v)$ its inward normal derivative at the endpoint $v$ of $e$.
\begin{theorem}
\label{mainthmquantumgraph}
The numbers $b_k$ in \eqref{eqnhighev Gepsilon} are the eigenvalues of the operator $-d^2/d\xi^2$ on the metric graph $G$ with edge lengths $2l_e$, defined on the space of functions $u$ which on each edge $e$ are smooth and take values in $\calN_e$ and which at each vertex $v$ satisfy the boundary conditions
\begin{align}
\label{eqn bc1}
(u_e(v))_{e\sim v} &\in (+1)-\text{eigenspace of } S_v(0)\\
\label{eqn bc2}
(\partial_n u_e(v))_{e\sim v} &\in (-1)-\text{eigenspace of } S_v(0).
\end{align}
\end{theorem}
\begin{remarks}\mbox{}
\begin{enumerate}
\item If, for some edge $e$, the lowest eigenvalue on $Y_e$ is bigger than $\nu$, then $\calN_e=\{0\}$, so $u_e$ has to be identically zero, which means that the edge $e$ may be omitted from $G$. In other words, only the 'thickest' edges contribute substantially to the eigenvalue $\mu_k(\epsilon)$, for small $\epsilon$.
\item
If $Y_e$ is connected then $\calN_e$ is one-dimensional, so it can be identified with $\C$. But in general there is no canonical eigenfunction, so this identification is not canonical. However, if all $Y_e$  are the same connected Riemannian manifold (for example, in the case of an $\epsilon$-neighborhood of an embedded graph, where they are balls) one may choose the same basis element in each $\calN_e$, so one may think of $u$ as a complex valued function and $V_v\cong\C^{\deg(v)}$ where $\deg(v)$ is the degree of $v$.
\item
Two special cases deserve to be mentioned:
\begin{itemize}
\item Dirichlet conditions at the vertex $v$ correspond to $S_v(0)=-\Id$. In particular, if $S(0)=-\Id$ then the limit quantum graph is completely decoupled (Dirichlet boundary value problem on each edge separately, without interaction between edges).
\item
If all $\calN_e$ are identified with $\C$ and $S_v(0)$ has $(+1)$-eigenspace spanned by the vector $(1,1,\dots,1)$ then \eqref{eqn bc1} says that $u$ is continuous at $v$ and \eqref{eqn bc2} (which is equivalent to $(\partial_n u_e(v))_{e\sim v}\perp (1,\dots,1)$) that the sum of normal derivatives of the $u_e$ vanishes at $v$. This is often called the 'Kirchhoff boundary condition'.
\end{itemize}
\item
Theorems \ref{maintheorem_Gepsilon} and \ref{mainthmquantumgraph} show clearly the two ingredients that determine the leading behavior of eigenvalues:
\begin{itemize}
\item
The determination of $S(0)$ and of the $L^2$-eigenvalues $\tau_k$ is a transcendental problem, depending on the vertex and edge manifolds (for example, in the situation of the $\epsilon$-neighborhood of an embedding, the angle at which edges meet).
\item
Given $S(0)$, the $b_k$ are determined solely by the combinatorics of the underlying metric graph.
\end{itemize}
In the special case of pure Neumann boundary conditions and of Dirichlet conditions with 'small' vertex manifolds  the leading behavior of eigenvalues is determined by the metric graph alone, see Theorem \ref{thm special cases}. That is, it is independent of the manifolds $X_v,Y_e$. This may be seen as reason why the general case is harder to analyze.
These cases were known previously, see \cite{KucZen:CSMSCG}, \cite{RubSch:VPMCTSI}, \cite{ExnPos:CSGLTM} and \cite{Pos:BQWGDBCDC}.

It is known that 'usually' (for example in the embedded situation) there are $L^2$-eigenvalues $\tau<\nu$, so they can not be neglected. See, for example, \cite{SchRavWyl:QBSCUSCW}, \cite{AviBesGirMan:QBSOG}.

\end{enumerate}
\end{remarks}

Using Theorem \ref{mainthmquantumgraph} one can show that for generic geometric data (for example, an open dense set of Riemannian metrics on the vertex and edge manifolds) one has $S(0)=-\Id$, that is decoupled Dirichlet conditions for the quantum graph at all vertices. This will be pursued in a separate paper.

\subsection{Outline of the proof of the Main Theorem}
We now give an outline of the proof of Theorem \ref{maintheorem1}, where for simplicity we assume that all $l_e=1$:
Let $x$ be a coordinate on the cylindrical part of $X^\infty$ and of $\XNG$, measuring length along the cylinder axis (going from $0$ to $N$ from both ends of the cylinder axis).

The proof consists of two steps: First, we use the spectral data on $X^\infty$ to construct approximate eigenfunctions on $X^N_G$ for large $N$ and conclude the existence of eigenvalues and eigenfunctions as claimed. Second, we show that all eigenfunctions on $X^N_G$ are obtained this way.

For clarity we assume in this outline that all multiplicities (i.e.\ dimensions of $\calE_p$ and $\calP_z$) are equal to one and that $\tau_p<\nu$ for all $p$ (no embedded eigenvalues). We will call a quantity 'very small' if it is exponentially small as $N\to\infty$, i.e. $O(e^{-cN})$ for some $c>0$.

Throughout, eigenvalues on $\XNG$ will be denoted $\mu$, with $\mu=\nu+\beta^2$ if $\mu\geq\nu$, and spectral values for $\Xinfty$ by $\lambda$, with $\lambda=\nu+\alpha^2$ if $\lambda\geq\nu$.

{\bf Case I:} Eigenvalues  $\mu<\nu-c$: Here, everything is quite straightforward.
\begin{list}{}{}
\item
{\em First step:} If $U$ is an eigenfunction on $\Xinfty$ with eigenvalue $\lambda=\tau_p<\nu$ then $U$ decreases exponentially in $x$. So if we simply cut off $U$ smoothly near $x=N$, making it zero near $x=N$, then we get a function on $X^N_G$ which satisfies the eigenfunction equation up to a very small error. A simple spectral approximation lemma, Lemma \ref{lemmaSAL}, shows that
$-\DeltaXNG$ has an eigenvalue very close to $\lambda$. A spectral gap argument gives a similar approximation for the eigenfunction.
\item
{\em Second step:} This procedure can be reversed: If $u$ is an eigenfunction on $\XNG$ with eigenvalue $\mu<\nu-c$ then it decreases exponentially in $x$, so cutting it off near $x=N$ yields a function on $\Xinfty$ which satisfies the eigenfunction equation up to a very small error. Since $-\DeltaXinfty$ has purely discrete spectrum near $\mu$, it follows by the spectral approximation lemma that it has an eigenvalue very close to $\mu$. Therefore, there can be no eigenfunctions on $\XNG$ in addition to those constructed in the first step.
\end{list}

{\bf Case II:} Eigenvalues $\mu\in (\nu+e^{-cN},\lambdamax]$: A basic observation is that any eigenfunction $u$ on $\XNG$ or generalized eigenfunction on $\Xinfty$ with eigenvalue $\lambda\in(\nu,\nu_1)$ can on the cylindrical part be written as $\Pi u+\Piperp u$, where $\Pi u$, called the leading part, is the first mode in the $Y$-direction. $\Pi u$ is of the form
\begin{equation}\label{eqn leading part}
\Pi u = e^{-i\alpha x}\phi + e^{i\alpha x}\psi,\quad \phi,\psi\in V,\quad \alpha:=\sqrt{\lambda-\nu}
\end{equation}
and  $\Piperp u$ is exponentially decreasing in $x$.
\begin{list}{}{}
\item
{\em First step:}
The scattering solutions $E=E_{\alpha,\phi}$ have leading part of the form \eqref{eqn leading part} with $\psi=S(\alpha)\phi$. Since $\Pi E$ is not decaying as $x\to\infty$, we should only cut off $\Piperp E$ near $x=N$, in order to obtain a small error term when constructing an approximate eigenfunction on $\XNG$ from $E$. Therefore, we must require $\Pi E_{\alpha,\phi}$ to satisfy matching conditions at $x=N$ that make it a smooth function on the cylindrical part of $\XNG$.
A short calculation shows that this is equivalent to the equation
\begin{equation}
\label{eqn M def}
M(\alpha,\phi)=0,\quad \text{where } M(\alpha,\phi):= \left[\Id-e^{i\alpha N2L}\sigma S(\alpha)\right] \phi \in V.
\end{equation}
Perturbation theory gives functions $Z_i^\rho$ and spaces $\calP_i^\rho$ so that the solutions of this equation are $\alpha=\frac 1{N} Z_i^\rho(\frac 1{N})$  and $\phi\in \calP_{i}^\rho (\frac1{N})$. So for these $\alpha,\phi$ the function obtained from $E=E_{\alpha,\phi}$ by cutting of $\Piperp E$ near $x=N$ satisfies the eigenvalue equation on $\XNG$ with a very small error. Therefore, one gets eigenvalues as in \eqref{eqn alpha-formel}.
\item
{\em Second step:} It remains to show that for large $N$ all eigenvalues $\mu\in(\nu+e^{-cN},\lambdamax]$ on $\XNG$ are obtained in this way. This is the theoretically most demanding part of the proof. Let $u$ be an eigenfunction of $-\DeltaXNG$ with eigenvalue $\mu=\nu+\beta^2, \beta>e^{-cN/2}$. An argument directly analogous to the case $\mu<\nu$ won't work since it only yields that $-\DeltaXinfty$ has nonempty spectrum near $\mu$, which we know anyway.
Rather, we need to show that $u$ is very close to some $E=E_{\alpha,\phi}$ with $(\alpha,\phi)$ satisfying \eqref{eqn M def} and with $\alpha$ very close to $\beta$. The closeness of $u,E$ would follow from closeness of the leading parts $\Pi u,\Pi E$ since they control the full function, and this is proved in two steps: First, we prove an elliptic estimate which reflects the essentials of scattering theory in a {\em compact } elliptic problem and implies that eigensolutions whose $\Piperp$ part decays exponentially for $x\leq N$ (as $\Piperp u$ does) are close to scattering solutions (whose $\Piperp$ part decays exponentially for $x\to\infty$) with the same eigenvalue, see Lemmas \ref{lemma basic ell estimate} and \ref{lemma eigenfcn est}. This is a stable version of the existence and uniqueness of scattering solutions on $\Xinfty$. So we obtain some $E_{\beta,\phi_0}$ very close to $u$. In a second step it is shown that $(\beta,\phi_0)$ must be very close to a solution of \eqref{eqn M def}: Since $u$ satisfies the matching conditions at $x=N$ and $u-E$ is very small, $E$ almost satisfies these conditions, so $M(\beta,\phi)$ is very small; therefore, all that is needed is a stable version of the analysis of \eqref{eqn M def} (which, however, is quite non-trivial).
\end{list}

{\bf Case III:} Eigenvalues $\mu\in[\nu-c,\nu+e^{-cN}]$: While the first step (construction of approximate eigenfunction) presents no new difficulties, the second step (proof that all eigenfunctions are obtained) is quite delicate. One difficulty is that the representation \eqref{eqn leading part} is not valid at $\alpha=0$ (it does not give all scattering solutions). There is a straight-forward replacement, however. More serious is showing that any eigenvalue $\mu\in[\nu-c,\nu)$ must actually be in $(\nu-e^{-cN},\nu)$. This again requires a delicate stability analysis of the matching condition.

\medskip

Special care needs to be taken when the multiplicities are not equal to one, since it is not enough just to construct eigenvalues, one also needs sufficiently many. A useful tool here is the notion of distance between subspaces of a vector space which we recall in Section \ref{subsec dist subspaces}.

If one is only interested in the case of fixed $k$ as in Theorem \ref{maintheorem_Gepsilon} then some of the proofs can be simplified considerably since the functions $Z_i^\rho$ remain separated for different $i$ then. See, for example, Corollary \ref{cor bounded i}.
In particular the proof of Lemma \ref{lemma stability} simplifies considerably in this case, and Theorem \ref{thm unitary perturb} is not needed.

\subsection{Outline of the paper}
In Section \ref{sec setup} we introduce the setup precisely and some notation. Section \ref{sec basics analysis} introduces the basic analytic tools, most importantly separation of variables and distance between subspaces, as well as a basic spectral approximation lemma. In Section \ref{sec scatt} we recall the facts from scattering theory that we need. In Section \ref{sec scatt matching} we analyze which scattering solutions satisfy the matching conditions, in particular, equation \eqref{eqn M def}. The basic elliptic estimate and some consequences are proven in Section \ref{secestef}. Theorem \ref{maintheorem1} and improvements of it are proven in Section \ref{sec main proofs}, and Theorem \ref{mainthmquantumgraph} in Section \ref{sec quantum graph}, where we also discuss some special cases.

In the Appendix we collect some basic results on one-parameter families of unitary operators which are needed in the analysis of \eqref{eqn M def}.  Their proofs can be found in \cite{Gri:MUF}.

\subsection{Related work}
\label{subsec previous}
As already mentioned, the Neumann problem was treated in \cite{Col:MPVPNNL}, \cite{FreWen:DPGAP}, \cite{KucZen:CSMSCG}, \nocite{KucZen:ASNLTD} \cite{RubSch:VPMCTSI}, \cite{ExnPos:CSGLTM}. For Dirichlet boundary conditions, Post \cite{Pos:BQWGDBCDC} derived the first two terms of \eqref{eqnhighev Gepsilon} in the case of 'small' vertex neighborhoods, see Theorem \ref{thm special cases}.
In the recent preprint  \cite{MolVai:LONTFSNT} Molchanov and Vainberg study the Dirichlet problem and show that, in the context of Theorem \ref{maintheorem_Gepsilon}, the $\mu_k(\epsilon)-\epsilon^{-2}\nu$ converge to eigenvalues of the quantum graph described in Theorem \ref{mainthmquantumgraph}; this was conjectured in \cite{MolVai:SSNTFSDA}, where also some results on the scattering theory on non-compact graphs are obtained. However, their statements are unclear as to whether the multiplicities coincide; also, they do not consider the effect of $L^2$ eigenvalues on $X^\infty$ or uniform asymptotics for large $k$.
In \cite{AlbCacFin:CSLTQW} a related model is considered. The method in the previously cited papers is to compare quadratic forms or to show resolvent convergence of some sort, and in all cases only the leading asymptotic behavior is obtained.

Problems of the same basic analytic structure (cylindrical neck stretching to infinity, attached to fixed compact ends; usually with $G$ consisting of one edge only) were studied by various authors in the context of global analysis, where they occur in a method to prove gluing formulas for spectral invariants. In their study of analytic torsion (related to the determinant of the Laplacian), Hassell, Mazzeo and Melrose \cite{HasMazMel:SFMWCCT} gave a very precise description of the resolvent (in the case of closed manifolds, i.e.\ no boundary, but admitting edge neighborhoods which are not precisely cylindrical, just asymptotically), including its full asymptotic behavior as $\epsilon\to 0$, using ideas of R.\ Melrose's 'b-calculus', a refined version of the pseudodifferential calculus.  More direct approaches were used by Cappell, Lee and Miller \cite{CapLeeMil:SAEOMDILES} and by M\"uller \cite{Mue:EIMWB} in the study of the $\eta$-invariant (for the Dirac operator instead of the Laplacian) and by Park and Wojciechowski \cite{ParWoj:ADZDDNO}. The author and Jerison \cite{GriJer:AEPD} prove a special case of Theorem \ref{maintheorem1} (where $\XNG$ is a plane domain obtained by attaching a long rectangle to a fixed domain, which is required to have width at most the width of the rectangle) by a different method (matched asymptotic expansions) and use it to prove a result about nodal lines of eigenfunctions; for this, one needs to know the asymptotic behavior to second order (i.e. one order more than written explicitly in \eqref{eqnhighev Gepsilon}).

In the context of this literature, the main purpose of the present paper is to give a mathematically rigorous yet straightforward derivation of the limiting problem on the graph, allowing various boundary conditions, in a way that admits generalization to similar problems. For example, there are straightforward generalizations to higher order operators, systems and Schr\"odinger operators with potential, as long as one has a product type structure along the edges.
We use existence of the scattering matrix for manifolds with cylindrical ends  as a 'black box'. The main technical problem is the proof that all eigenfunctions are obtained by the given construction; here the neighborhood of the threshold $\nu$ requires a special effort (this is sometimes referred to as the problem with 'very small eigenvalues' since the eigenvalues are exponentially close to $\nu$).

\medskip

{\bf Notation:}
As usual constants $c,C>0$ may have different values at each occurence (unless otherwise stated).
They depend on the data (the graph, the edge lengths, the compact manifolds), but not on $N$.
The $L^2$ scalar product on a space $X$ is denoted by $\langle u,v\rangle_X=\int_X u\overline v$ and the $L^2$ norm by $\|\cdot\|_X$.

\section{The setup: Combinatorial and geometric data} \label{sec setup}
The following data are given:
\begin{itemize}
\item
{\bf Combinatorial data:}

A finite graph $G=(V,E)$, with vertex set $V=V(G)$,  edge
set $E=E(G)$. Loops and multiple edges are allowed. Thus,
$E$ may be thought of as a multiset of unordered pairs of
vertices.
If the vertex $v$ is adjacent to an edge $e$, we write $v\sim e$.
A {\em half-edge} is a
pair $(v,e)$ with $v\sim e$, and we denote by
$\Ebar$ the set of half edges, except that for a loop $e$ at a vertex $v$ the element
$(v,e)$ appears twice in $\Ebar$ (so formally $\Ebar$ is really a multiset).
Sometimes we denote
a half-edge by $\ebar$ if it arises from the edge $e$.
The {\em neighborhood} of a vertex $v$ is the (multi-)set of half-edges incident
to it.

It will be useful to think of $G$ as the disjoint union of the
vertex neighborhoods, with the 'ends' of the half-edges
glued together appropriately. The glueing may be encoded
by a map $\sigma$, where
\begin{align}\label{sigma}
\sigma: \Ebar & \to \Ebar \\
(v_1,e) & \mapsto (v_2,e)
\end{align}
if the edge $e$ connects $v_1$ and $v_2$ (for a loop $e$, $\sigma$ maps one copy of $(v,e)$ to the other).
%(BETTER: For simplicity, omit loops, do a footnote once on how to include them.)
$\sigma$ is an involution, that is $\sigma^2=\Id$,  and has no fixed points.

We also assume that a positive number $l_e$ is  given for each edge $e$,
to be thought of as half the length of $e$. Denote the shortest half edge length by
$$\lmin = \min_e l_e.$$

\item
{\bf Geometric data:}

\begin{itemize}
\item
To each vertex $v$ a compact Riemannian manifold $(X_v,g_{X_v})$ with piecewise smooth boundary%
\footnote{for an exact definition of this, see for example \cite{HasZel:QEBVE}; for our purposes one may assume that the data are such that the spaces $X^N$ defined below are smooth manifolds with boundary.},
of dimension $n$
\item
to each edge $e$ a compact Riemannian manifold $(Y_e,g_{Y_e})$ with smooth boundary, of dimension $n-1$; for a half edge $\ebar$ corresponding to an edge $e$, we set $Y_\ebar:= Y_e$;
\item
to each half-edge $(v,e)$ an isometry $h_{v,e}$ (gluing map) from $Y_e$ to a subset of the boundary of $X_v$; we assume that $(X_v,g_{X_v})$ is of product type near $h_{v,e}(Y_e)$, see below; also, for each $v$ the sets $h_{v,e}(Y_e)$ are assumed to be disjoint (no overlaps of different edges);
\item
partitions of the boundary of each $Y_e$ and of the part of the boundary of each $X_v$ which is not
in the image of any of these isometries, into two pieces denoted by indices D and N (for Dirichlet and Neumann boundary conditions); we assume sufficient regularity of this decomposition so that the boundary value problems formulated below (before \eqref{eqnlaplaceproduct}) are well-posed.

More generally, one may give a pair of non-negative functions $a,b$ (of sufficient regularity to make the problem below well-posed) on the boundary of each $X_v$ (outside the gluing part) and each $Y_e$, with $a,b$ never vanishing simultaneously, to define Robin boundary conditions, see below. The D/N decomposition corresponds to $a,b$ being characteristic functions of a partition of these boundaries into two parts.
\end{itemize}
\end{itemize}
From this data, we define Riemannian manifolds $X^N,$ $0<N\leq\infty$, with piecewise smooth boundary, as follows: First, for each half-edge $\ebar=(v,e)$ we attach a cylinder with cross section $Y_\ebar$,
$$ Z^N_\ebar := [0,N)\times Y_\ebar$$
to $X_v$, using the isometry $h_{v,e}$.
Thus, for each vertex $v$
we get a manifold with piecewise smooth boundary
$$ X_v^N := \left[ X_v \cup %Z_v^\infty,\quad Z_v^\infty =
\bigcup_{\ebar:\,\ebar\sim v} Z^N_\ebar \right] / \{h_{v,e}\}_{v\sim e}$$
where the quotient means that each $\{0\}\times Y_\ebar$ is
identified isometrically with a subset of the boundary of $X_v$.

On $Z^N_{\ebar}$ we put the cylindrical Riemannian metric
\begin{equation}
\label{eqnmetriconZ}
g_{Z^N_\ebar}=l_e^2\, dx^2 + g_{Y_e}.
\end{equation}
By assumption, $(X_v,g_{X_v})$ is of product type near $Y_\ebar':=h_{v,e}(Y_e)$, which means that a neighborhood of $Y_\ebar'$ in $X_v$ is isometric to $(0,\epsilon)\times Y_\ebar'$ with the product metric, for some $\epsilon>0$. This ensures that the metrics on $X_v$ and $Z^N_\ebar$ define a smooth Riemannian metric on $X^N_v$ (it also fixes the smooth structure on $X^N_v$).

Let
$$X^N := \bigcup_{v\in V} X_v^N$$ (disjoint union) and denote by
$$ Y := \bigcup_{\ebar\in\Ebar} Y_\ebar,\quad Z^N:=[0,N)\times Y = \bigcup_{\ebar\in\Ebar} Z^N_\ebar$$
the cross section resp.\ the cylindrical part of $X^N$.

For finite $N$ the pieces $X_v^N$ are now glued together as prescribed by the graph $G$ to give
the $N$-neighborhood, $X^N_G$, of $G$. More precisely,
$\sigma:\Ebar\to\Ebar$ induces bijections, also denoted $\sigma$, $Y_\ebar\to Y_{\sigma(\ebar)}$ (since both of these cross sections are just copies of the same $Y_e$), and therefore $\{N\}\times Y_\ebar\to \{N\}\times Y_{\sigma(\ebar)}$ for each $\ebar\in \Ebar$, and then
\begin{equation}
\label{eqn def XNG}
X^N_G := \overline{X^N} / \sigma
\end{equation}
where $\overline{X^N}$ is analogous to $X^N$, except that $x<N$ is replaced by $x\leq N$.
In other words, $X^N_G$ is the union of the $X_v$ and cylinders $[0,2N]\times Y_e$
for each edge $e$ of $G$, glued together according to the structure of $G$.
(But our $x$ coordinate will run between $0$ and $N$ from both ends of the interval.)
We also write
$$ Z^N_G := \overline{Z^N} / \sigma.$$

Clearly, the D/N decomposition of the boundaries of each  $X_v$ and of each $Y_e$
gives a corresponding decomposition of the boundary of each $X^N$, of $X^N_G$ and $Y$, and $Z^N$.
The Riemannian metrics define Laplace operators $\Delta_{X^N},\Delta_{X^N_G},\Delta_Y,\Delta_{Z^N}$ on these spaces, for which we impose Dirichlet boundary conditions on the D part of the boundary and Neumann
boundary conditions on the N part (but no boundary condition at the boundary piece $\{0\}\times Y$ of $Z^N$).
More generally, one may consider Robin boundary conditions, $a u + b\partial_n u=0$, where the functions $a,b$ on $\partial X^N$ (etc.) are induced by those on the boundaries of $X_v$, $Y_e$, by making them independent of the cylinder coordinate $x$.

By \eqref{eqnmetriconZ},
\begin{equation}\label{eqnlaplaceproduct}
\Delta_{Z^N_\ebar}=l_e^{-2}\frac{\partial^2}{\partial x^2} + \Delta_{Y_e}.
\end{equation}

\begin{remark}
The Riemannian metric \eqref{eqnmetriconZ} expresses the fact that the cylindrical part $Z^N_\ebar$ has length $Nl_e$. Instead, one could use the coordinate $\xi= x l_e$, then one would get the more standard form
\begin{equation} \label{eqnLaplace-in-xi}
g_{Z^N_\ebar}=d\xi^2+g_{Y_e},\quad \Delta_{Z^N_e} = \frac{\partial^2}{\partial\xi^2} + \Delta_{Y_e}, \quad \xi=x l_e,
\end{equation}
and $0\leq\xi<l_e N$ on $Z^N_\ebar$. The $\xi$-coordinate is more convenient for some calculations, but in the $x$-coordinate $Z^N_\ebar$ is given by the same range of $x$, $0\leq x<N$, for all $\ebar$. We switch between these coordinates as is convenient.
\end{remark}

\section{Basics of the analysis} \label{sec basics analysis}
\subsection{Notation}
Since $Y$ is compact, the Laplacian $\Delta_Y$ has compact resolvent and therefore discrete spectrum. Let
$\nu_0<\nu_1<\dots$ be the eigenvalues of $-\Delta_Y$, with finite dimensional
eigenspaces $V_0,V_1,\dots$. We also write
$$ \nu:=\nu_{0},\quad V:= V_{0}.$$
Let $\Pi$ be the orthogonal projection to $V$ in $L^2(Y)$, and $\Pi^\perp = \Id_{L^2(Y)} -\Pi$.

Let $N\in (0,\infty]$.
We decompose any $u\in L^2_{\loc}(X^N)$ in its 'vertical' $\Pi$ and $\Pi^\perp$ components, over the cylindrical part $Z^N$:
$$ u_{|Z^N} = \Pi u + \Pi^\perp u$$
where $ (\Pi u)(x) := \Pi(u(x,\cdot))$ and similarly for $\Piperp$. Here and throughout the paper we identify functions on $Z^N$ with functions on $[0,N)$ whose values are functions on $Y$. In particular,
$$ \Pi u: [0,N)\to V,\quad \Piperp u:[0,N)\to V^\perp.$$

Since $Y=\bigcup_{\ebar\in\Ebar}Y_{\ebar}$ and $\Delta_{Y}$ acts on each $Y_\ebar$ separately, we have
\begin{equation}
\label{eqn splitting}
L^2(Y)=\bigoplus_{\ebar\in\Ebar} L^2(Y_{\ebar}),\quad
V=\bigoplus_{\ebar\in\Ebar} V_{\ebar},\quad V_{\ebar} := \ker (\Delta_{Y_{\ebar}}+\nu).
\end{equation}
We write elements $\phi\in L^2(Y)$  as $\phi = (\phi_{\ebar})_{\ebar\in\Ebar}$ with $\phi_{\ebar}\in L^2(Y_\ebar)$. If
$\phi\in V$ then $\phi_\ebar\in V_\ebar$ $\forall \ebar$.

The data define two important linear maps $L,\sigma:L^2(Y)\to L^2(Y)$ which restrict to linear maps $L,\sigma:V\to V$:

The map $L$ is diagonal with respect to the splitting \eqref{eqn splitting} and encodes the edge lengths
\begin{equation}
\left.
\begin{array}{rrl}
L:&L^2(Y) &\to L^2(Y)\\
L:&V      &\to V
\end{array}
\right\}
,\
 (\phi_{\ebar}) \mapsto (l_{e}\phi_{\ebar}).
\end{equation}

The map $\sigma$ is defined by the involution $\sigma:\Ebar\to\Ebar$ encoding the graph structure:
\begin{equation}
\left.
\begin{array}{rrl}
\sigma :&L^2(Y) &\to L^2(Y)\\
\sigma :&V      &\to V
\end{array}
\right\}
,\ (\phi_\ebar) \mapsto (\phi_{\sigma(\ebar)}).
\end{equation}
We denote
\begin{equation}
\label{eqn V^pm def}
V^\pm := \text{ the }\pm 1\text{ eigenspace of }\sigma.
\end{equation}
Since $\sigma$ is a self-adjoint involution, the decomposition $V=V^+\oplus V^-$ is orthogonal.
For $\phi\in V$ we denote the corresponding decomposition $\phi=\phi^++\phi^-$.

By \eqref{eqnlaplaceproduct} and \eqref{eqnLaplace-in-xi}, we have
\begin{align}\label{eqn-laplace-on-Z}
\Delta_{Z^N} &= L^{-2} \frac{\partial^2}{\partial x^2} + \Delta_Y\\
  & = \bigoplus_{k=0}^\infty \left( L^{-2}\frac{d^2}{dx^2} - \nu_k \right) = \bigoplus_{k=0}^\infty \left( \frac{d^2}{d\xi^2} - \nu_k \right)
      \label{eqn-laplace-decomposed}
\end{align}
with respect to the decomposition $L^2(Y)=\bigoplus_{k=0}^\infty V_k$.

\subsection{Matching conditions at $x=N$}
We regard functions on $X^N_G$ as functions on $X^N$ satisfying suitable matching conditions at $x=N$. For solutions of the eigenfunction equation we get:
\begin{lemma} \label{lemmabc}
Let $u$ be an eigenfunction of  $-\Delta_{X^N}$ with eigenvalue $\lambda$.

Then $u$ defines an eigenfunction of $-\DeltaXNG$ if and only if
$u$ extends smoothly to $x=N$ and
\begin{equation}
\label{eqn matching u}
u^-=0,\quad (\partial_\xi u)^+=0,\quad\text{ at }x=N.
\end{equation}
The upper $\pm$ refer to the $\sigma$-decomposition \eqref{eqn V^pm def}.
\end{lemma}
For $\alpha>0$ let
\begin{equation}
\label{eqn def MC alpha}
\begin{gathered}
\BC_\alpha(u) := \Pi u^-_{x=N} + \frac i\alpha \Pi(\partial_\xi u)^+_{x=N} \in V.
\end{gathered}
\end{equation}
Then  $\BC_\alpha(u)=0$ if and only if $\Pi u$ satisfies the matching conditions at $x=N$.
The particular scaling in \eqref{eqn def MC alpha} is motivated by the calculation \eqref{eqn BCEcalc}.
% old version: \left( (I-\sigma) u_{x=N}, (I+\sigma)(\partial_\xi u)_{x=N} \right) \in L^2(Y)\times L^2(Y).
\begin{proof}
This is just the fact that the solution to a second order elliptic partial differential equation can be continued across a hypersurface, as a solution, if and only if at the hypersurface it is continuous and its normal derivatives match. Set
$\Phi=u_{x=N}$, $\Psi=(\partial_\xi u)_{x=N}$. Then $\Phi^-=\frac12(\Phi-\sigma\Phi)$ and $\Psi^+=\frac12(\Psi+\sigma\Psi)$, so \eqref{eqn matching u} means $\Phi_\ebar  = \Phi_{\sigma(\ebar)}$, $\Psi_\ebar = -\Psi_{\sigma(\ebar)}$ for all $\ebar$, that is, continuity of $u$ and of $\partial_\xi u$ at $x=N$ in $\XNG$.
\end{proof}

For a function on $X^N$, when we write $u_{x=N}$ we always assume that $u$ extends smoothly to $x=N$.

\subsection{Separation of variables}
The following simple lemma is basic to all the analysis.

\begin{lemma}\label{lemmasepvar}
Let $0<N\leq\infty$. Let $u\in C^2(Z^N)$ satisfy $(\Delta_{Z^N} + \lambda)u = 0$.
\begin{itemize}
\item[(a)] The leading part of $u$ has the form
\begin{align}
 \Pi u & = e^{-\sqrt{\nu-\lambda}\xi}\phi + e^{\sqrt{\nu-\lambda}\xi}\psi
                             & \text{ if } \lambda<\nu, \label{eqn1}\\
 \Pi u & = e^{-i\alpha\xi}\phi + e^{i\alpha\xi}\psi
                             & \text{ if } \lambda>\nu,\ \lambda=\nu+\alpha^2  \label{eqn2}\\
 \Pi u & = \phi + \xi\psi & \text{ if } \lambda=\nu, \label{eqn2a}
\end{align}
with $\phi,\psi\in V$.
(Replace $\xi$ by $xL$ to express this in terms of $x$.)
\item[(b)]
If $\lambda<\nu_1$ then
\begin{equation} \label{eqn3}
 \Pi^\perp u = \sum_{k=1}^\infty e^{-\sqrt{\nu_{k}-\lambda}\xi}\phi_{k} + e^{\sqrt{\nu_{k}-\lambda}\xi}\psi_{k}, \quad \phi_{k},\psi_{k}\in V_{k}
\end{equation}
\item[(c)]
$\phi,\psi$ in (a) and $\phi_{k},\psi_{k}$ in (b) are uniquely determined by $u$. If $N=\infty$ and $u$ is polynomially bounded as $\xi\to\infty$ then $\psi=0$ in \eqref{eqn1} and $\psi_{k}=0$
for all $k$ in \eqref{eqn3}.
\item[(d)]
Assume $N<\infty$ and $u$ extends to a solution on $Z^N_G$, or $N=\infty$ and $u$ is polynomially bounded.
Then there is a constant $C$ such that for all $M\leq N$ we have
\begin{align}
\|u_{x=M}\|_Y + \|(\partial_\xi u)_{x=M} \|_Y &\leq C e^{-cM}\|u\|_{Z^1}\label{expdecaymu<nu}\\
\intertext{if $\lambda<\nu,\text{ where } c= \sqrt{\nu-\lambda}\,\lmin$, and}
\|\Pi^\perp u_{x=M}\|_Y + \|(\Pi^\perp \partial_\xi u)_{x=M}\|_Y &\leq C e^{-cM}\|\Pi^\perp u\|_{Z^1}
\label{expdecaymu>nu}
\end{align}
if $\lambda <\nu_1,\text{ where } c= \sqrt{\nu_1-\lambda}\,\lmin$.
\item[(e)]
If $N<\infty$ and $u$ is an eigenfunction on $Z^N_G$ then, for the representation in \eqref{eqn3},
\begin{equation}
\label{eqn Bperp small}
\|\sum_{k=1}^\infty (\nu_k)^{1/4}\sqrt{\nu_k-\lambda}\psi_k \|_Y \leq C e^{-2cN} \|\Piperp u_{x=0}\|_Y
\end{equation}
if $\lambda<\nu_1$ for any $c<\sqrt{\nu_1-\lambda}\lmin$.
\end{itemize}
\end{lemma}

\begin{proof}
For $k\in\N_0$ let $\Pi_k:L^2(Y) \to V_k$ be the projection and let $u_k=\Pi_k u:[0,N)\to V_k$. By \eqref{eqn-laplace-decomposed}, $u_k$ satisfies the differential equation
\begin{equation}\label{eqn-eveqn-uk}
(\frac{d^2}{d\xi^2}+\lambda-\nu_k)u_k=0,
\end{equation}
and this has the solutions given by the formulas in (a) and by the summands in (b). Clearly, $u_k$ and therefore $\phi,\psi,\phi_k,\psi_k$ are determined by $u$, and if $u$ is polynomially bounded then so is $u_k=\Pi_k u$ for each $k$, so (c) follows. This immediately gives (d) in the case $N=\infty$.
If $N<\infty$ and $u$ extends to a solution on $Z^N_G$ then $u$ may be regarded as function $[0,2N]\to L^2(Y)$. We now use the coordinate $x\in[0,2N]$. If $\nu_k>\lambda$ then one can express the solution of \eqref{eqn-eveqn-uk} by its values at $a,b$, for any $0<a<b<2N$: Write $\beta_k=\sqrt{\nu_k-\lambda}L$, then
$$u_k(x) = \frac{\sinh (x-a)\beta_k}{\sinh (b-a)\beta_k} u_k(b) +
\frac{\sinh (b-x)\beta_k}{\sinh (b-a)\beta_k} u_k(a). $$
For any $a\in (0,1)$ and $b\in (2N-1,2N)$ one obtains easily at $x=M\leq N$:
$$ \|u_k(M)\|_{V_k} \leq C e^{-\sqrt{\nu_k-\lambda}\lmin M}\left( \|u_k(a)\|_{V_k} + \|u_k(b)\|_{V_k}\right).$$
Summing over $k=0,1,2,\dots$ in the case $\lambda<\nu$ yields the estimate on $u_{x=M}$ in \eqref{expdecaymu<nu}. The estimate on the derivative and estimate \eqref{expdecaymu>nu} are obtained similarly.

(e) follows similarly to (d), using $\Piperp u _{x=0}= \sum_k \phi_k + \psi_k$, $\Piperp u_{x=2N} = \sum_k e^{-\sqrt{\nu_k-\lambda}N}\phi_k + e^{\sqrt{\nu_k-\lambda}N } \psi_k$ and $u_{x=2N} = \sigma u_{x=0}$.
\end{proof}
It is clear from the proof that the behavior of solutions is different for $\lambda\geq\nu_1$. For the purpose of this paper, we will always consider $\lambda<\nu_1$. Define $\lambdamax,\alphamax$ by
\begin{equation}
\label{eqn def lambdamax}
\lambdamax = \nu+\alphamax^2,\quad \lambdamax\in(\nu,\nu_1)\text{ arbitrary},
\end{equation}
then the estimates will always be uniform for $\lambda\leq\lambdamax$.

We define the 'boundary data' of a function $u$ on $Z^N$ as
\begin{equation}
\label{eqn def bd data}
u^0 := \Pi u_{|x=0},\quad u^1 := \Pi (\partial_\xi u)_{|x=0}.
\end{equation}
Clearly, if $u$ is an eigenfunction then $(u^0,u^1)$ determine $\Pi u$ uniquely:
Instead of the representation \eqref{eqn1}, \eqref{eqn2}, \eqref{eqn2a} for the leading part of an eigenfunction we use the basis
\begin{equation*}
\calC_\lambda (\xi) =
\begin{cases}
\cos \sqrt{\lambda-\nu}\xi & (\lambda>\nu) \\
1                          & (\lambda=\nu) \\
\cosh\sqrt{\nu-\lambda}\xi & (\lambda<\nu)
\end{cases}
,\quad
\calS_\lambda(\xi) =
\begin{cases}
\frac{\sin \sqrt{\lambda-\nu}\xi}{\sqrt{\lambda-\nu}} & (\lambda>\nu) \\
\xi                          & (\lambda=\nu) \\
\frac{\sinh\sqrt{\nu-\lambda}\xi} {\sqrt{\nu-\lambda}} & (\lambda<\nu)
\end{cases}
\end{equation*}
of solutions to $(d_\xi^2 + \lambda-\nu)v=0$. This is more useful than the exponential basis for $\lambda$ near $\nu$ since it depends analytically on $\lambda$ even at $\lambda=\nu$. For a solution $u$ of $(\Delta_{Z^N}+\lambda)u=0$ we have
\begin{equation}
\label{eqn Pi u from u0u1}
 \Pi u = \calC_\lambda u^0 +  \calS_\lambda u^1.
\end{equation}

\subsection{Distance between subspaces and spectral approximation}
\label{subsec dist subspaces}
It will be convenient to use the notion of distance between subspaces. Although it is standard, we recall its definition and basic properties here.

If $V,W$ are closed subspaces of a Hilbert space $H$ then define
$$ \dist(V,W) := \inf_{u\in V\setminus\{0\}} \frac{\dist(u,W)}{\|u\|},$$
where $\dist(u,W)=\inf\{\|u-w\|:\, w\in W\}$.
$\dist$ is not a distance function since it is asymmetric (it does satisfy the triangle inequality, however).

It is elementary to check the following properties (see \cite{HelSjo:MWSCLI}, for example):

\begin{lemma}
\label{lemma dist}
\begin{enuma}
\item $ \dist(V,W) = \|P_WP_V-P_V\| $
where $P_V,P_W$ are the orthogonal projections to $V,W$, respectively.
\item $\dist(V,W)\leq 1,\ \text{ and } <1 \text{ iff } \Pi_{W|V}:V\to W\text{ is injective.}$

Here, $\Pi_{W|V}$ is the restriction of $P_W$ to $V$. In particular $\dim V\leq \dim W$ if $\dist(V,W)<1$.
\item
If  $\dist(V,W)<1$ and $\dist (W,V)<1$ then $P_{V|W}:V\to W$ is an isomorphism and $\dist(V,W)=\dist(W,V)$. In this case we write
$$\distsymm(V,W) := \dist(V,W) = \dist (W,V).$$
\item
If $V_i$, $i\in I$, are pairwise orthogonal then
\begin{equation}
\label{eqn dist orthogonal}
\dist(\bigoplus_i V_i, W) \leq \sum_i \dist (V_i,W).
\end{equation}
\item
If $u_i$, $i=1,\dots,K$ are pairwise orthogonal and $\dist(\Span\{u_i\},W)<(1+\dim W)^{-1}$ for each $i$ then $K\leq\dim W$.
\end{enuma}
\end{lemma}
(Proof of (e), for example: Let $K'=\dim W$. If $K>K'$ then $\dist(\Span\{u_1,\dots,u_{K'+1}\},W)<(K'+1)/(K'+1)=1$, so $K'+1<K'$ by (d) and(b), a contradiction.)

We will use the following standard spectral approximation lemma. We include a proof for completeness.
\begin{lemma}[Spectral Approximation Lemma]\label{lemmaSAL}
Let $A$ be a selfadjoint operator in a Hilbert space $\calH$. For an interval $I\subset \R$ let $\Eig_I(A)$ be the spectral  subspace of $A$ corresponding to the spectral interval $I$.

Let $\lambda_0\in\R$, $\epsilon,\delta>0$. If $W\subset\ \Dom(A)$ is a linear subspace satisfying
\begin{equation} \label{eqnest1}
\|(A-\lambda_{0})u\| \leq \epsilon \|u\| \quad \forall u\in W
\end{equation}
then
\begin{equation}
\label{eqn distWEig}
\dist (W,\Eig_{(\lambda_0-\delta,\lambda_0+\delta)}(A)) \leq \frac\epsilon\delta.
\end{equation}
In particular, if $\epsilon<\delta$ then $\dim\Eig_{(\lambda_0-\delta,\lambda_0+\delta)}(A) \geq \dim W$.
\end{lemma}
An important consequence is that existence of  spectral gaps implies good approximation of eigenfunctions: If $A$ has no spectrum in $\{\lambda:\epsilon\leq|\lambda-\lambda_0|<\delta\}$ then $W$ is $\epsilon/\delta$-close to $\Eig_{(\lambda_0-\epsilon,\lambda_0+\epsilon)}(A)$.
\begin{proof}
We may assume $\lambda_0=0$. Let $A=\int \lambda\, dE_\lambda$ be the spectral resolution of $A$.
Let $P_\delta$ be the the projection to $\Eig_{(-\delta,\delta)}(A)$. If $u\in\Dom(A)$ is arbitrary then
\begin{equation}
\label{eqn Pdelta estimate}
\delta^2\|u-P_\delta u\|^2 \leq  \|Au\|^2
\end{equation}
since
$ \delta^2 \int_{|\lambda|\geq\delta} d(u,E_\lambda (u))
\leq \int_{|\lambda|\geq\delta} \lambda^2 d(u,E_\lambda (u))
\leq \int_\R \lambda^2 d(u,E_\lambda (u)) $. Therefore, if $w\in W$ then $\|w-P_\delta w\|\leq \frac\epsilon\delta \|w\|$, which was to be shown.
The last claim follows from Lemma \ref{lemma dist}b).
\end{proof} % SAL

\section{Scattering theory background} \label{sec scatt}
Here we collect the facts from scattering theory that we need. A good reference for this material is \cite{Gui:TSQVB} (this may not be the earliest source).

Scattering theory is the spectral theory
of an elliptic operator on a non-compact space which has a 'simple' structure at
infinity, that is, is asymptotically equal to an operator on a space for
which the spectral decomposition may be written down fairly explicitly.
Among the central goals of scattering theory are
the determination of the absolutely continuous spectrum
and of generalized eigenfunctions correponding to it.

We will need fairly explicit spectral information on the operator  $-\Delta_{X^\infty}$. Since $X^\infty$ is cylindrical at infinity (that is, outside  a compact subset), we take as explicit model $-\Delta_{Z^\infty}'$, where the prime means that we impose Neumann boundary conditions at
$x=0$ (Dirichlet would be equally possible), in addition to
the boundary conditions coming from the D/N decomposition (resp. Robin data) of  $\partial Y$,  in order to make the operator essentially self-adjoint.
The spectral theory of $-\Delta_{Z^\infty}'$ is easy to obtain, using separation of variables, i.e. \eqref{eqn-laplace-decomposed}. We use the $\xi$-coordinate on $[0,\infty)$ for simplicity. It will be replaced by $xL$ at the end.

Since $-\frac{d^2}{d \xi^2}$ on $[0,\infty)$, with Neumann condition at $\xi=0$, has absolutely continuous spectrum
$[0,\infty)$, with generalized eigenfunctions $\cos \alpha \xi$ corresponding
to the spectral parameter $\alpha^2$, $\alpha\geq 0$, the decomposition \eqref{eqn-laplace-decomposed} shows that the absolutely continuous spectrum of the model is
$$\specabs (-\Delta_{Z^\infty}') = \bigcup_{k\geq 0} [\nu_k,\infty) = [\nu_0,\infty)=[\nu,\infty),$$
and that the functions
$$ %E_0^{(k)}(\phi,\lambda) =
\cos (\sqrt{\lambda-\nu_k}\,\xi)\phi\quad\text{for }k,\phi\text{ satisfying }
\quad \nu_k\leq\lambda,\quad\phi\in V_k$$
span the generalized eigenfunctions of $-\Delta_{Z^\infty}'$ with eigenvalue $\lambda$. That is, each such function $U$ satisfies $-\Delta_{Z^\infty}'U=\lambda U$ and is polynomially bounded as $\xi\to\infty$, and any function with these properties is in the linear span of these functions.

The scattering theory for spaces with cylindrical ends shows that this picture
carries over to $-\Delta_{X^\infty}$, except for the possible appearance of discrete spectrum,
and a 'phase shift' and exponentially decaying error term in the generalized
eigenfunctions.

First, we have the description of the spectrum:

\begin{theorem}
(\cite{Gui:TSQVB}) \label{thmscatt1}
The operator $-\Delta_{X^\infty}$ is essentially self-adjoint on $C_0^\infty(X^\infty)$. Its unique self-adjoint extension (still denoted $-\Delta_{X^\infty}$) has the following properties:
\begin{enuma}
\item
The pure point spectrum is a discrete subset of $[0,\infty)$.
\item
The singularly continuous spectrum is empty.
\item
The absolutely continuous spectrum is $[\nu,\infty)$.
\end{enuma}
\end{theorem}
Discreteness includes finite multiplicity.

We now discuss the generalized eigenfunctions (or 'scattering solutions') $E$,
\begin{equation}\label{Eeqn}
-\Delta_{X^\infty}E=\lambda E,\quad \lambda=\nu+\alpha^2
\end{equation}
for spectral values $\lambda\in[\nu,\nu_1)$. One could also consider values $\nu\geq\nu_1$, this would mean using higher scattering matrices.

By Lemma \ref{lemmasepvar}, $E_{|Z^\infty} = \Pi E + \Piperp E$ where $E$ has the explicit form \eqref{eqn2} or \eqref{eqn2a} and $\Piperp E$ is a sum of exponentially decreasing and exponentially increasing terms. By polynomial boundedness, the latter must vanish.

\begin{theorem}[\cite{Gui:TSQVB}] \label{thmscatt2}\
\begin{enuma}
\item ($\lambda>\nu)$
Assume that $\lambda\in(\nu,\nu_1)$ is not an $L^2$-eigenvalue of $-\DeltaXinfty$.
For each $\phi\in V$ there is
a unique bounded solution $E=E_{\alpha,\phi}$ of \eqref{Eeqn}
satisfying (with $\alpha=\sqrt{\lambda-\nu}$)
\begin{equation}\label{scattsoln}
\Pi E_{\alpha,\phi} = e^{-i\alpha \xi}\phi + e^{i\alpha \xi}\phi'\quad\text{ for some }\phi'\in V.
\end{equation}
\item (Scattering matrix)
This defines a linear map
$$S(\alpha): V\to V,\quad \phi\mapsto\phi',$$
called the {\em scattering matrix}.  $S(\alpha)\in \End(V_0)$ is holomorphic in $\alpha$ and extends meromorphically to $\alpha\in\C\setminus\{\alpha\in\R:\ |\alpha|\geq \sqrt{\nu_1-\nu}\}$, holomorphic for real $\alpha$.
Furthermore:
\begin{enumi}
\item $S(\alpha)$ is unitary for $\alpha$ real.
%\item $S(\alpha)$ is hermitian for $\alpha$ imaginary. (NOT NEEDED?)
\item $S(\alpha)S(-\alpha)=I$ for all $ \alpha$.
\item $S(0)$ is an involution, and $S'(0)$ commutes with $S(0)$.
\end{enumi}
\item ($\lambda=\nu$)
The space $\calS_\alpha=\{E_{\alpha,\phi}:\,\phi\in V\}$ depends holomorphically on $\alpha$ for $\alpha$ as in a) and extends holomorphically as in b). For $\alpha=0$ it is described as follows:

Let $V_\pm$ be the $\pm 1$ eigenspaces of $S(0)$. For all $\Phi_+\in V_+$, $\Psi_-\in V_-$ there is a unique generalized eigenfunction $E=E_{0,\Phi_+,\Psi_-}\in\calS_0$ satisfying
\begin{equation}
\label{scattsolnu}
\Pi E_{0,\Phi_+,\Psi_-} = \Phi + \xi\Psi_-, \quad \Phi = \Phi_+ + \frac i2 S'(0)\Psi_-,
\end{equation}
Here, $\frac i2 S'(0)\Psi_-\in V_-$.
\end{enuma}
The functions in (a) and (c) are all the generalized eigenfunctions, up to addition of possible $L^2$-eigenfunctions. %Furthermore, they satisfy the estimates
%\begin{equation}\label{eqnpiperpEest}
%\| \Pi^\perp E_{\alpha,\phi} \|_{Z^\infty} \leq C\|\phi \|\ \text{ resp. }\ \| \Pi^\perp E_{0,\Phi_+,\Psi_-} \|_{Z^\infty} \leq C \|\Phi_+\| + \|\Psi_-\|.
%\end{equation}
\end{theorem}
Note that if $\nu+\alpha^2$ is an $L^2$-eigenvalue, the function $E_{\alpha,\phi}$ resp. $E_{0,\Phi_+,\Psi_-}$ is not uniquely determined by fixing its leading part, since the leading part of an $L^2$-eigenfunction is zero. However, it is determined by the additional requirements that it lie in $\calS_\alpha$ and that this space depends continuously on $\alpha$ (for $\alpha\neq0$ this means that $\alpha\mapsto E_{\alpha,\phi}$ depends continuously on $\alpha$).
\begin{proof}
This is mostly standard. Note that b)iii) follows from ii) by setting $\alpha=0$, and from differentiating ii) at $\alpha=0$, which gives $S'(0)S(0)-S(0)S'(0)=0$. This implies that $S'(0)$ preserves $V_\pm$, and therefore the last claim in c). Equation  \eqref{scattsolnu} will be explained below.
%Let us also prove the formula in c) since it is less well-known:
%Let $\calS_\alpha$ be the subspace of $C^\infty((0,\infty),V)$ given by the $\Pi E$ parts of generalized eigenfunctions $E$ for $\lambda=\nu+\alpha^2$, for $0\leq\alpha<\sqrt{\nu'-\nu}$. By general principles, this is a $\dim V$-dimensional space for each $\alpha$, varying continuously with $\alpha$. For $\alpha>0$ we have $\calS_\alpha=\{e_{\alpha,\phi}:\ \phi\in V\}$, $e_{\alpha,\phi} := e^{-i\alpha \xi}\phi + e^{i\alpha \xi} S(\alpha)\phi$.
%
%Given $\Phi_+\in V_+, \Psi_-\in V_-$ consider $\phi=\phi(\alpha) = \frac12 \Phi_+ + \frac i{2\alpha} \Psi_-$. Writing $S(\alpha)=S(0) + \alpha S'(0) + O(\alpha^2)$,
%we see
%$$ S(\alpha)\phi = \frac12 \Phi_+ - \frac i{2\alpha}\Psi_- + \frac i2 S'(0)\Psi_- + O(\alpha),\quad \alpha\to 0$$
%and then
%$$ e_{\alpha,\phi} = \frac {e^{-i\alpha\xi} + e^{i\alpha\xi}}2 \Phi_+
%                + \frac{e^{-i\alpha \xi}-e^{i\alpha \xi}}{2\alpha}i\Psi_- + e^{i\alpha\xi} \frac i2 S'(0)\Psi_- + O(\alpha),$$
%and this converges to $\Phi_+  + \xi \Psi_- + \frac i2S'(0)\Psi_-$ as $\alpha\to 0$.
%Therefore, $\calS_0=\lim_{\alpha\to0}\calS_{\alpha}$ includes the space of the functions $E_{0,\Phi_+,\Psi_-}$, and since both spaces have dimension $\dim V$, they must be equal.
\end{proof}

In our context, we use the variable $x$, where $\xi = xL$. Therefore
\begin{align}\label{eqnEalpha}
\Pi E_{\alpha,\phi} &= e^{-i\alpha x L }\phi + e^{i\alpha x L} S(\alpha)\phi\\
\Pi E_{0,\Phi_+,\Psi_-} &= (\Phi_+ + \frac i2 S'(0)\Psi_-) + x L\Psi_-,
\label{eqnEPhiPsi}
\end{align}
respectively.

Theorem \ref{thmscatt2} can be reformulated in terms of 'scattering subspaces'; we will use this since it allows a more uniform treatment of the cases $\alpha\neq0$ and $\alpha=0$:

Recall the notation $u^0=\Pi u_{x=0}$, $u^1=\Pi \partial_\xi u_{x=0}$.
Define the 'scattering subspace'  for $\alpha\in [-\alphamax,\alphamax]$ by
\begin{equation}
\label{eqn def scatt subspace}
\calL_{\alpha} := \{(E^0,E^1):\, (\DeltaXinfty + \nu+\alpha^2)E=0,\ E \text{ polynomially bounded}\} \subset V\times V.
\end{equation}
Since $\Pi E$  determines and is determined by $(E^0,E^1)$, part a) of Theorem \ref{thmscatt2} may be reformulated as follows:
\begin{equation}
\label{eqn calL for alpha positive}
\calL_\alpha = \{\left(\;(I+S(\alpha))\phi, -i\alpha (I-S(\alpha))\phi\;\right):\ \phi\in V \}\quad\text{ for }\alpha\neq 0\text{ not an }L^2-\text{eigenvalue},
\end{equation}
and to each $(E^0,E^1)\in\calL_\alpha$ there is a unique scattering solution $E$ with eigenvalue $\nu+\alpha^2$.

In particular, $\calL_\alpha$ is a  $\dim V$-dimensional subspace of $V\times V$ for these $\alpha$; since $\calL_\alpha=\Pi\calS_\alpha$ it extends analytically to all $|\alpha|<\sqrt{\nu_1-\nu}$. However, in general the continuation to $\alpha=0$ cannot be obtained by setting $\alpha=0$ on the right of \eqref{eqn calL for alpha positive} since this yields $V_+\times\{0\}$, a space of lower dimension (unless $S(0)=\Id$),
so this cannot be $\calL_0$.  In other words, the parametrization of $\calL_\alpha$ by $\phi$ does not extend uniformly to $\alpha=0$. Therefore, we will also need another parametrization which works (and is analytic) also at $\alpha=0$:

\begin{lemma}
\label{lemma scatt subspace}
For $\alpha>0$,
\begin{equation}
\label{eqn calL rescaled}
\begin{aligned}
\calL_\alpha = \{(E^0,E^1):\
& E^0 = \rho_+ + \tfrac i2 S'(0)\rho_- + \alpha R^0(\alpha)\rho \\
& E^1 = \rho_- - \tfrac 12 \alpha S'(0)\rho_- + \alpha^2 R^1(\alpha)\rho,\\
& \rho\in V\}
\end{aligned}
\end{equation}
for certain families $R^0(\alpha),R^1(\alpha)$ of  endomorphisms of $V$, depending analytically on  $\alpha$ for $|\alpha|\leq\alphamax$.

If $S(\alpha)=S(0)$ for all $\alpha$ then $\calL_\alpha=\{(\rho_+,\rho_-):\rho\in V\}$ for all $\alpha$.
\end{lemma}
\begin{proof}
Write $S(\alpha)=S(0)+\alpha T + \alpha^2 R(\alpha)$, $T:=S'(0)$, by Taylor's formula.  By \eqref{eqn calL for alpha positive}, $(E^0,E^1)\in V\times V$ is in $\calL_\alpha$ (for $\alpha>0$) iff, for some $\phi\in V$,
\begin{align*}
E^0&=(I+S(\alpha))\phi && = 2\phi_+ + \alpha T \phi_- + \alpha T \phi_+ + \alpha^2 R(\alpha)\phi\\
E^1&=-i\alpha(I-S(\alpha))\phi && = -2i\alpha \phi_- + i\alpha^2 T \phi_- + i\alpha^2 T \phi_+ + i\alpha^3 R(\alpha)\phi.
\end{align*}
Here we used $S(0)\phi=\phi_+-\phi_-$ and $(T\phi)_\pm=T\phi_\pm$. Write
\begin{equation}
\label{eqn rescale phi}
\rho_+ = 2\phi_+,\ \rho_- = -2i\alpha\phi_-,
\end{equation}
then this becomes \eqref{eqn calL rescaled}, with suitable $R^0,R^1$.
The last statement is clear from this derivation.
\end{proof}

In particular,
\begin{equation}
\label{eqn L0}
\calL_0= \{(\rho_+ + \frac i2 S'(0)\rho_-,\rho_-):\, \rho\in V\},
\end{equation}  and this explains \eqref{scattsolnu}.

\section{Scattering solutions and matching conditions}
\label{sec scatt matching}
In this section we analyze for which scattering solutions $E$ the leading part $\Pi E$ satisfies the matching condition at $x=N$.

\subsection{The case $\lambda=\nu$}
\begin{lemma}
\label{lemmabcEthreshold}
Let $\Phi_+\in V_+$, $\Psi_-\in V_-$. If $N>\frac1{2\lmin}\|S'(0)\|$ then $\Pi E_{0,\Phi_+,\Psi_-}$, given by \eqref{eqnEPhiPsi}, satisfies the matching conditions \eqref{eqn matching u} iff
\begin{equation}
\label{eqnbcEthreshold}
\sigma \Phi_+=\Phi_+,\quad \Psi_-=0.
\end{equation}
\end{lemma}
\begin{proof}
The matching conditions \eqref{eqn matching u} for $\Pi E_{0,\Phi_+,\Psi_-}$ are
\begin{equation}
\label{eqnbc1}
 (I-\sigma) (\Phi_+ + \frac i2 S'(0) \Psi_- + NL\Psi_-) = 0, \quad (I+\sigma)\Psi_-=0.
\end{equation}
This is clearly satisfied if \eqref{eqnbcEthreshold} holds. Conversely, if
\eqref{eqnbc1} holds then $\Phi_+ + \frac i2 S'(0) \Psi_- + NL\Psi_-$ and $\Psi_-$ are orthogonal, so
\begin{equation}
\label{eqn aaa}
 0=\langle\Phi_+ + \frac i2 S'(0) \Psi_- + NL\Psi_-,\Psi_-\rangle=\langle\frac i2 S'(0)\Psi_-,\Psi_-\rangle + N\langle L\Psi_-,\Psi_-\rangle,
\end{equation}
and Cauchy-Schwarz gives
$$ \|\frac i2 S'(0)\Psi_-\|\cdot \|\Psi_-\| \leq \frac12 \|S'(0)\|\cdot \|\Psi_-\|^2.$$
Since $\lmin \|\Psi_-\|^2 \leq (L\Psi_-,\Psi_-)$, this and \eqref{eqn aaa} imply $\Psi_-=0$ if $N>\frac1{2\lmin} \|S'(0)\|$, and then \eqref{eqnbc1} gives $\sigma \Phi_+=\Phi_+$.
\end{proof}

\subsection{The case $\lambda>\nu$}
First, we express the matching conditions in terms of the scattering matrix.
\begin{lemma} \label{lemmabcE}
Let $\alpha>0$, $\phi\in V$. Then $\Pi E_{\alpha,\phi}$, given by \eqref{eqnEalpha}, satisfies the matching conditions \eqref{eqn matching u} iff
\begin{equation} \label{eqnbcE}
(I-e^{i\alpha N2L}\sigma S(\alpha)) \phi = 0.
\end{equation}
\end{lemma}
\begin{proof}
Let $\psi=S(\alpha)\phi$. Then by \eqref{eqn def MC alpha}
\begin{equation}\label{eqn BCEcalc}
\begin{aligned}
\BC_\alpha(E_{\alpha,\phi}) &= [e^{-i\alpha NL} \phi + e^{i\alpha NL} \psi]^-
+ \frac i\alpha[-i\alpha e^{-i\alpha NL} \phi + i\alpha e^{i\alpha NL} \psi]^+
\\
&= e^{-i\alpha NL}\left([\phi + e^{2i\alpha NL}\psi]^- + [\phi - e^{2i\alpha NL}\psi]^+ \right) \\
&= e^{-i\alpha NL} (\phi - e^{2i\alpha NL}\sigma\psi).
\end{aligned}
\end{equation}
Here we used that $\sigma L=L\sigma$. (This reflects the fact that two half-edges corresponding to the same edge have the same length.) $\Pi E_{\alpha,\phi}$ satisfies the matching conditions iff $\BC_\alpha(E_{\alpha,\phi})=0$, so the claim follows.
\end{proof}

We now analyze the solutions $(\alpha,\phi)$ of equation \eqref{eqnbcE}, and in particular their asymptotic behavior as $N\to\infty$.
It is convenient to introduce the rescaled variable
$$z := \alpha N.$$
Let
$$ U(z,\alpha)=e^{iz2L}\sigma S(\alpha)$$
for $z>0$, where $\alpha \in[0,\alphamax]$.
We need to study the zero set
\begin{equation}
\label{eqn def Z}
Z = \{(z,\alpha)\in (0,\infty)\times [0,\alphamax]:\ \det(I-U(z,\alpha))=0\}
\end{equation}
and to each $(z,\alpha)\in Z$ the eigenspace $\ker(I-U(z,\alpha))$.
Intersecting $Z$ with the line $z=\alpha N$ then gives the solutions of \eqref{eqnbcE}.

The $U(z,\alpha)$ are unitary operators on $V$. Essential for the sequel is the 'monotonicity' in $z$:
\begin{equation}
\label{eqn Umonotone}
\frac1i \frac{\partial U}{\partial z} U^{-1} = 2L>0.
\end{equation}
We begin with the case $\alpha=0$, which should be regarded as the limit $N\to\infty$.

\begin{proposition}
\label{prop alpha=0}
The set of $z$ for which $(z,0)\in Z$ consists of a sequence $0<z_1<z_2<\dots\to\infty$. If $\calP_{z_i}$ denotes the kernel of $I-U(z_i,0)$ then
\begin{equation}\label{eqn asymptotics}
 \sum_{i: z_i < A} \dim \calP_{z_i} = \frac{\tr L}{\pi} A + O(1),\quad A\to\infty.
\end{equation}
\end{proposition}

This is obvious in the case $L=lI$, i.e.\ if all edges of the graph have equal lengths, since then
$\calP_{z}$ is just the eigenspace of $\sigma S(0)$ with eigenvalue $e^{-iz2l}$.
\begin{proof}
By monotonicity, equation \eqref{eqn Umonotone}, we can apply Lemma \ref{lemma counting ev} to the unitary family $z\mapsto U(z,0)$ with $D(z)=2L$, and this gives the result.
\end{proof}
The structure of $Z$ and the eigenspaces is then given as follows.
\begin{proposition}\label{propperturbedef}
Let $z_1,z_2,\dots$ be as in the previous proposition.
For each $i\in\N$ there are $r_i\in \N$ and pairwise different real analytic functions $z_i^\rho$, $\rho=1,\dots,r_i$, defined on  $[0,\alphamax]$, such that
\begin{equation}
\label{eqn initial}
z_i^\rho(0)=z_i \quad\text{for each } i,\rho
\end{equation}
 and
\begin{equation}
\label{eqn Z description}
Z = \bigcup_{i,\rho} \{(z_i^\rho(\alpha),\alpha):\ \alpha\in [0,\alphamax]\}.
\end{equation}
There is a constant $C_0$ such that
\begin{equation}
\label{eqn a' bound}
|\frac{d}{d\alpha}z_i^\rho(\alpha)| \leq C_0 \quad\text{ for all }i,\rho,\alpha.
\end{equation}

Furthermore, for each $i,\rho$ there is a real analytic family of orthogonal projections $P_{i}^\rho(\alpha)$ on $V$,
$\alpha\in [0,\alphamax]$, such that for each $(z,\alpha)\in Z$
\begin{equation}
\label{eqn proj for Z}
\sum_{i,\rho: z_i^\rho(\alpha)=z} P_i^\rho(\alpha) = \text{ the projection to }
\ker (I-U(z,\alpha)).
\end{equation}
\end{proposition}

Thus, $Z$ is the union of graphs of functions of $\alpha$ with bounded derivatives, see Figure \ref{figure2}. The 'non-linear eigenvalues' $z_i$ at $\alpha=0$ may bifurcate into various $z_i^\rho$ as $\alpha$ increases.  The sum in \eqref{eqn proj for Z} is over the various branches that meet at $(z,\alpha)$, so for almost all $(z,\alpha)$ it has only one term (in particular $P_i^\rho$ is uniquely determined).
\begin{figure}
\includegraphics{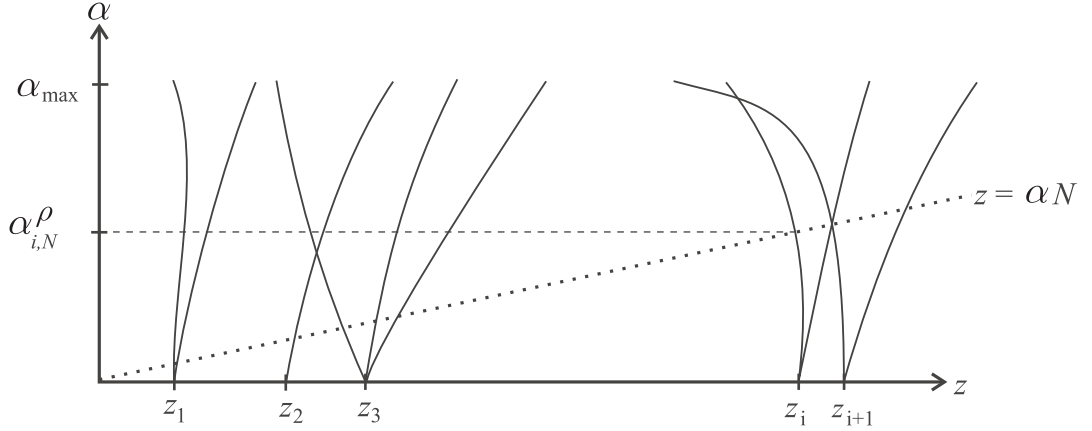}
\caption{The solutions of $\det(I-U(z,\alpha))=0$}
\label{figure2}
\end{figure}
\begin{proof}
Theorem \ref{thm unitary perturb} shows that $Z$ is everywhere locally a union of graphs of analytic functions, and these can be patched to functions on all of $[0,\alphamax]$. The theorem gives also the projections. Applying the theorem at $(z_i,0)$ gives \eqref{eqn initial}.

It remains to prove \eqref{eqn a' bound}.
If $z(\alpha)=z_i^\rho(\alpha)$ and $\phi(\alpha)\in\Ran P_i^\rho(\alpha)$ is normalized and chosen analytic in $\alpha$ then differentiating $U(z(\alpha),\alpha) \phi(\alpha) = \phi(\alpha)$ gives
$z' (\partial_z U) \phi + (\partial_\alpha U)\phi + U\phi' = \phi'$. Taking the scalar product with $\phi$ one obtains, since $\la U\phi',\phi\ra = \la\phi',U^{-1}\phi\ra=\la\phi',\phi\ra$, that $z'\la(\partial_z U)\phi,\phi\ra + \la(\partial_\alpha U)\phi,\phi\ra=0$ and hence, using $\phi=U^{-1}\phi$ and  \eqref{eqn Umonotone},
$$ z' = -\frac{\la e^{iz2L}\frac1i \sigma S'(\alpha)\phi,\phi\ra}{\la 2L\phi,\phi\ra},$$
which is uniformly bounded as claimed.
\end{proof}
We now study the solutions of \eqref{eqnbcE} and for this the intersection of $Z$ with the line $z=\alpha N$, for fixed $N$. The following is quite obvious from Figure \ref{figure2}.

\begin{theorem} \label{thmscattef1}
\begin{itemize}
\item[(a)]
Fix arbitrary $i,\rho$. Then there is a unique point
\begin{equation}
\label{eqn def alpha i rho N}
(\alpha_{i,N}^\rho,z_{i,N}^\rho) = \text{  intersection point of }
z = z_i^\rho(\alpha)\text{ and } z=\alpha N
\end{equation}
whenever $N\geq N_{0,i}:= C_0 + \frac{z_i}{\alphamax}$ with $C_0$ from \eqref{eqn a' bound}.

Furthermore, $z_{i,N}^\rho$ is real analytic in $\frac1N$, i.e.\ $z_{i,N}^\rho = Z_i^\rho(\frac1N)$ for a function $Z_i^\rho$ which is real analytic on $[0,(N_{0,i})^{-1}]$. Also,
$Z_i^\rho(0) = z_i$.
\item[(b)]
The pairs $(\alpha,\phi)$, $0<\alpha\leq\alphamax$, for which $\Pi E_{\alpha,\phi}$ satisfies the matching conditions \eqref{eqn matching u} are given by
\begin{align}
\alpha &= \alpha_{i,N}^\rho =\frac1N z_{i,N}^\rho \label{alphasol}\\
\phi &\in \calP_{\alpha,N}:=\bigoplus_{i,\rho:\,\alpha^{\rho}_{i,N}=\alpha} \Ran P_{i}^\rho(\alpha) \label{phisol}
\end{align}
for $i=1,2,\dots$ satisfying $N\geq N_{0,i}$ and $\rho=1,\dots,r_i$.
The $P_i^\rho$ are defined in Proposition \ref{propperturbedef}.
\end{itemize}
\end{theorem}
\begin{proof}
(a) Fix $i,\rho$. The $z$-coordinates of the intersection points \eqref{eqn def alpha i rho N} are the solutions of $z=z_i^\rho(tz)$, where $t=\frac1N$, that is, the zeroes of the function $d(z)=z-z_i^\rho(tz)$. We have $d(0)=-z_i<0$ and $d(\frac\alphamax t)= \frac\alphamax t - z_i^\rho(\alphamax) \geq 0$ for $t\leq \tmax:= (N_{0,i})^{-1}$ since $z_i^\rho(\alphamax)\leq z_i + C_0\alphamax$ by integration of the bound \eqref{eqn a' bound}. Therefore, $d$ has a zero $z$ for each $t\leq\tmax$. The zero $z$ is unique since
$d'(z)=1-t(z_i^\rho)'(tz)>0$ for $t<1/C_0$ (which is satisfied for $t\leq\tmax$).
Clearly, for $t=0$ the solution is $z=z_i$ by \eqref{eqn initial}, and the inverse function theorem gives the analytic dependence on $t$.

(b)
By Lemma \ref{lemmabcE}, $\Pi E_{\alpha,\phi}$ satisfies the matching conditions iff \eqref{eqnbcE} is satisfied. The claim then follows from Proposition \ref{propperturbedef}.
\end{proof}
A more precise analysis yields the uniform behavior of the functions $Z_i^\rho$ with respect to $i\to\infty$: $Z_i^\rho(t) = z_i + h_i^\rho(tz_i)$ for functions $h_i^\rho$ which vanish at zero and have bounds on their derivatives independent of $i$.
We will omit the proof since we don't need it. Instead, we prove the following slightly weaker consequence of Theorem \ref{thmscattef1} about special regimes of solutions:
\begin{theorem}
\label{thm alphaleadingterm}
Given $\alpha_0\in [0,\alphamax]$ and $\delta>0$, $C>0$, the values of $\alpha$ in \eqref{alphasol} in the interval
$|\alpha-\alpha_0| < CN^{-\delta}$ satisfy
\begin{equation}
\label{eqn alphaleadingterm b}
\alpha = \frac{b_i^\rho}N + O(\frac 1 {N^{1+\delta}})
\end{equation}
where $b_i^\rho = z_i^\rho(\alpha_0)$, that is, $b=b_i^\rho$ are the solutions of
\begin{equation}
\label{eqn birho}
\det(I-e^{ib2L}\sigma S(\alpha_0))=0,
\end{equation}
and only the range $|b_i^\rho - N\alpha_0| < C N^{1-\delta}$ is considered. The constant implied in \eqref{eqn alphaleadingterm b} is uniform in $i,\rho,N$ and $\delta$.
\end{theorem}
In particular, for $\alpha_0=0,$ $\delta=1$ we get: The solutions $\alpha$ with $\alpha<C/N$ are of the form
\begin{equation}
\label{eqn alphaleadingterm a}
\alpha = \frac{z_i}N + O(\frac1{N^2})
\end{equation}
This is already clear from Theorem \ref{thmscattef1}a).
\begin{proof}
Let $t=1/N$. If $\alpha=\alpha_i^\rho(t)$ satisfies $|\alpha-\alpha_0|<Ct^\delta$ then $|z_i^\rho (\alpha) - z_i^\rho(\alpha_0)| < C_0C t^\delta$ by \eqref{eqn a' bound}. From $\alpha=tz_i^\rho(\alpha)$ it then follows that $|\alpha-tb_i^\rho| < C_0C t^{\delta+1}$, that is, \eqref{eqn alphaleadingterm b}. Finally, this estimate and $|\alpha-\alpha_0|<Ct^\delta$ imply $|tb_i^\rho - \alpha_0| < Ct^\delta (1+C_0t)$, and this yields the restricted range for $b_i^\rho$.
\end{proof}

We will also need the following stable version of Theorem \ref{thmscattef1}b).
\begin{theorem}
\label{thm stability}
Let $N>0$. Assume $\alpha\in (0,\alphamax)$ and $\phi\in V$ are such that   $\Pi E_{\alpha,\phi}$ satisfies the matching conditions up to an error
\begin{equation}
\label{eqn BCsmall}
 \| \BC_\alpha(E_{\alpha,\phi}) \|_{V} \leq \delta \| \phi \|_{V},
\end{equation}
for some $\delta>0$.
\begin{enuma}
\item
Then there is $i$, $\rho$ with
\begin{align}
|\alpha-\alpha_{i,N}^\rho| & \leq \frac2\lmin \frac\delta{N}.
\label{diffestimate1}
\end{align}
\item
Furthermore, there is $C>0$ such that, if $\delta < C^{-1}$ and $N>C$, then the sum $\calP_{\alpha,\delta,N}=\bigoplus_\beta \calP_{\beta,N}$, with $\beta$ ranging over $|\alpha-\beta|\leq N^{-1}\sqrt{\delta}$, is direct and, if $P$ is the projection to $\calP_{\alpha,\delta,N}$, then
\begin{equation}
\label{eqn diffestimate2}
\|\phi-P\phi\|_V \leq \delta^{1/(2+2\dim V)}\,\|\phi\|_V
\end{equation}
\end{enuma}
\end{theorem}
\begin{proof}
By \eqref{eqn BCEcalc} we have
$e^{i\alpha NL} \BC_\alpha(E_{\alpha,\phi})=\left(I - e^{i\alpha N2L}\sigma S(\alpha)\right)\phi,$
so \eqref{eqn BCsmall} implies $\|(I - e^{i\alpha N2L}\sigma S(\alpha))\phi\| \leq \delta\|\phi\|$. Write $z=\alpha N$ and
\begin{equation}
\label{eqn Uperturbed def}
U(z)=e^{iz2L}\sigma S(\frac z N),
\end{equation}
then this reads
\begin{equation}
\label{eqn  eqn I-U small}
\|(I-U(z))\phi\| \leq \delta\,\|\phi\|.
\end{equation}
Now $z\mapsto U(z)$ is a monotone unitary family in the sense of \eqref{eqn monotone}, \eqref{eqn dminmax def} since, with a prime denoting derivative in $z$,
$$ \frac1i U'U^{-1} = 2L + \frac 1 {iN} U S^{-1}S' U^{-1}$$
where $U$ is evaluated at $z$ and $S$ at $z/N$, and for $N$ sufficiently large and $\alpha=z/N$ bounded this is positive with bounds as in \eqref{eqn dminmax def}, where $\dmin$ can be taken as $\lmin$.

Therefore, we can apply Lemma \ref{lemma stability} with $\epsilon= \delta$. \eqref{diffestimate1} follows from \eqref{eqn dist to Z} (remember $\alpha=z/N$) and \eqref{eqn diffestimate2} from \eqref{eqn dist to sum of W}.
\end{proof}

\section{Elliptic estimates} \label{secestef}
We first state some fairly standard elliptic estimates on the 'compact part'
$$X^0 := \bigcup_{v\in V} X_v$$
of $X^\infty$. Since we consider the scattering theory results in Section \ref{sec scatt} as a black box in this article, we derive them from those results. In a more thorough and systematic treatment, they could be derived directly from the theory of elliptic boundary value problems and then used to derive the scattering theory results. However, the boundary value problem is slightly non-standard since it involves a non-local (pseudo-differential) boundary operator.

The boundary $\partial X^0$ splits in two parts: The part where the cylinders $Z^N$ are attached, which we may identify with $Y$ (which we sometimes also call $\{x=0\}$), and the complement $\partial X^0 \setminus Y$. At the latter, we have boundary conditions given by the D/N decomposition (resp. Robin data). At $Y$ we will now impose boundary conditions motivated from the scattering theory.

For $u\in C^\infty(X^0)$ and $\lambda<\nu_1$ let
$$\Bperp(\lambda)u := \Piperp\partial_\xi  u_{|Y} - Q(\lambda) (\Piperp u_{|Y}),$$
where for $\phi_k\in V_k$, $k\in\N$,
$$ Q(\lambda) (\sum_{k=1}^\infty \phi_k) := - \sum_{k=1}^\infty \sqrt{\nu_k - \lambda} \phi_k,$$
if the sums converge. $Q(\lambda)$ is the $\Piperp$ part of the Dirichlet to Neumann operator for exponentially decreasing solutions of $(\Delta_{Z^\infty}+\lambda)\utilde = 0$, see \eqref{eqn3}.
Thus, $ \Bperp(\lambda)u=0$ iff $\Piperp\utilde$ has no exponentially increasing part, where $\utilde$ is the unique function on $Z^\infty$ satisfying $(\Delta_{Z^\infty}+\lambda)\utilde=0$ and having the same value and normal derivative at $x=0$ as $u$ at $Y$.

Consider the operator $\Delta_{X^0}$ with domain $\Dom(\Delta_{X^0})\subset H^2(X^0)$ defined by the D/N boundary conditions (resp. Robin data) at $\partial X^0\setminus Y$.
In order to obtain a selfadjoint extension of $\Delta_{X^0}$ we need in addition to impose boundary conditions at $Y$. In addition to the condition $\Bperp(\lambda)u=0$ we need a condition involving $u^0:=\Pi u_{|Y}$, $u^1:= \Pi \partial_\xi u_{|Y}$. It is well-known and easy to check that selfadjoint boundary conditions correspond to subspaces $\calL\subset V\times V$ which are {\em Lagrangian}, i.e.\ such that $\dim\calL=\dim V$ and
$$ (u^0,u^1), (v^0,v^1)\in \calL \Longrightarrow \langle u^0,v^1\rangle - \langle v^0,u^1\rangle = 0.$$
Thus, for $\calL$ Lagrangian and any $\lambda$ the operator $\Delta_{X^0}$ is selfadjoint on the domain $\{u\in\Dom(\Delta_{X^0}:\, \Bperp(\lambda)=0,\, (u^0,u^1)\in\calL\}$.

We have the following elliptic estimates.
\begin{lemma}
\label{lemma basic ell estimate}
Let $\alpha\in [0,\alphamax]$. Assume $\lambda=\nu+\alpha^2$ is not an $L^2$ eigenvalue of $-\DeltaXinfty$. Let $\calL_\alpha$ be the scattering subspace \eqref{eqn def scatt subspace}, and let $\calL_\alpha'$ be a Lagrangian subspace of $V\times V$ such that $\calL_\alpha\cap\calL'_\alpha=\{0\}$.
Also, let $P_{\calL_\alpha,\calL'_\alpha}:V\times V\to \calL_\alpha$ be the projection along $\calL_\alpha'$.
There is a constant $C$ so that if $u\in C^\infty(X^0)$ satisfies the D/N (resp. Robin) boundary conditions at $\partial X^0\setminus Y$ and
\begin{equation}
\label{eqn ell system}
\begin{aligned}
(\Delta + \lambda)u &= f & \text{ in } X^0\\
\Bperp(\lambda)u     &= g & \text{ at } Y\\
P_{\calL_\alpha,\calL'_\alpha}(u^0,u^1)           & = h
\end{aligned}
\end{equation}
then
\begin{equation}
\label{eqn basic ell estimate}
\|u\|_{H^2(X^0)} \leq C (\; \|f\|_{L^2(X^0)} + \|g\|_{H^{1/2}(Y)} + \|h\|_{V\times V}\;).
\end{equation}
If $\alpha$ is varied and the $\calL_\alpha'$ depend continuously on $\alpha$ then the constant $C$ can be chosen independent of $\alpha$.
\end{lemma}
See Section \ref{subsec L2 eigenvalues} for a replacement in case $\nu+\alpha^2$ is an $L^2$-eigenvalue.
\begin{proof}
If $u$ is a solution of the homogeneous problem, i.e. $f=0$, $g=0$, $h=0$, then $g=0$ implies that $u$ extends to a solution of $(\Delta+\lambda)u=0$ on $X^\infty$ with bounded $\Piperp$ part and hence that $u$ is a scattering solution. Then $(u^0,u^1)\in\calL_\alpha$ by definition of $\calL_\alpha$. $h=0$ implies $(u^0,u^1)\in\calL'_\alpha$, hence $(u^0,u^1)=0$ since $\calL_\alpha\cap\calL'_\alpha=\{0\}$ and therefore $u\equiv 0$ by uniqueness of scattering solutions.

Therefore, the map $u\mapsto ((\Delta+\lambda)u,\Bperp(\lambda)u, P_{\calL_\alpha,\calL'_\alpha}(u^0,u^1)),\  \Dom(\Delta_{X^0}) \to L^2(X^0)\times H^{1/2}(Y)\times (V\times V)$ is injective. It is also surjective, since for $g=0$, $h=0$ one has a solution $u$ for any $f$ by the Fredholm alternative (since the operator $\Delta_{X^0}+\lambda$ is selfadjoint on the domain which consists of those $u$ satisfying homogeneous boundary conditions, has closed range and is injective), and arbitrary $g,h$ can be removed by replacing $u$ by $u-v$, where $v$ is any function in $\Dom(\Delta_{X^0})$ satisfying $\Bperp(\lambda)v=g$, $P_{\calL_\alpha,\calL'_\alpha}(v^0,v^1)=h$. $v$ exists by standard arguments, since $X^0$ is of product type near $Y$ by assumption. Since $\Dom(\Delta_{X^0})$ is complete with the $H^2$-norm the open mapping theorem gives \eqref{eqn basic ell estimate}.

To show that $C$ can be chosen independent of $\alpha$ it suffices to show that it can be chosen locally uniformly with respect to $\alpha$. This can be proved as follows: Fix $\alpha_0$ and let $C_0$ be the constant for $\alpha=\alpha_0$. Suppose $u$ satisfies \eqref{eqn ell system} for some $\alpha$ near $\alpha_0$. This can be rewritten $(\Delta+\lambda_0)u=f+(\lambda_0-\lambda)u,\ \Bperp(\lambda_0)u = g + (\Bperp(\lambda_0)-\Bperp(\lambda))u, P_{\calL_{\alpha_0},\calL'_{\alpha_0}} (u^0,u^1) = h + (P_{\calL_{\alpha_0},\calL'_{\alpha_0}}-P_{\calL_\alpha,\calL'_\alpha})(u^0,u^1)$.
Estimate \eqref{eqn basic ell estimate} with these data  yields $
\|u\|_{H^2(X^0)} \leq C_0 (\; \|f\|_{L^2(X^0)} + \|g\|_{H^{1/2}(Y)} + \epsilon\|u\|_{H^2(X^0)})$, where $\epsilon\to 0$ as $\alpha\to\alpha_0$ (independently of $u$), since all operators on the right in \eqref{eqn ell system} depend continuously on $\alpha$. For $\epsilon<1/2C_0$ the last term can be absorbed into the left hand side, and the claim follows.
\end{proof}

Using \eqref{eqn basic ell estimate} with $u=E$ a scattering solution we get
\begin{equation}
\label{eqn E bounded by E0E1}
\|E\|_{H^2(X^0)}\leq C\|(E^0,E^1)\|_{V\times V}\quad\text{ for any scattering solution }E.
\end{equation}

When applying Lemma \ref{lemma basic ell estimate} we will need the following estimate which shows that the exponentially increasing part, $\Bperp(\lambda) u$, of an eigenfunction on $\XNG$ is very small.

\begin{lemma}
\label{lemma eigenfcn est}
Let $(\DeltaXNG + \mu)u=0$, $\mu<\nu_1$. Then
\begin{equation}
\label{eqn Bperp estimate}
\|\Bperp(\mu) u\|_{H^{1/2}(Y)} \leq Ce^{-cN}\|u\|_{H^2(X^0)}.
\end{equation}
\end{lemma}
\begin{proof}
If $\Piperp u= \sum_{k=1}^\infty e^{-\sqrt{\nu_k-\mu}\xi}\phi_k + e^{\sqrt{\nu_k-\mu}\xi}\psi_k$ with $\phi_k,\psi_k\in V_k$ then $\Bperp(\mu)u = \sum_{k=1}^\infty \sqrt{\nu_k-\mu} \psi_k$.
Since $\Delta_Y$ is an elliptic operator of order $2$, with eigenvalues $\nu_k$, it follows that
$\|\Bperp(\mu)u\|_{H^{1/2}(Y)} \leq \|\sum_{k=1}^\infty ({\nu_k})^{1/4} \sqrt{\nu_k-\mu} \psi_k\|_{L^2(Y)}$, and then \eqref{eqn Bperp small} together with the trace theorem gives the claim.
\end{proof}

We will need that scattering solutions for different spectral values whose leading parts satisfy the matching conditions are almost orthogonal on $X^N$:
\begin{lemma}
\label{lemma E almost orthogonal}
Let $\alpha,\alpha'\in A_N$ and $\phi\in \calP_{\alpha,N},\phi'\in \calP_{\alpha',N}$. If $\alpha\neq\alpha'$ then the restrictions of $E=E_{\alpha,\phi}, E'=E_{\alpha',\phi'}$ to $X^N$ are almost orthogonal, i.e.
$$ |\la E,E'\ra_{X^N}| \leq Ce^{-cN}\|E\|_{X^N}\|E'\|_{X^N}.$$
\end{lemma}
\begin{proof}
With $\lambda=\nu+\alpha^2,\ \lambda'=\nu+(\alpha')^2$ we have by Green's formula
\begin{align*}
(\lambda-\lambda')\la E,E'\ra_{X^N} & = \la\lambda E,E'\ra_{X^N} - \la E,\lambda'E'\ra_{X^N} = -\la\Delta E,E'\ra_{X^N} + \la E,\Delta E'\ra_{X^N}\\
                         & = \la E_{x=N},\partial_\xi E'_{x=N}\ra_Y - \la\partial_\xi E_{x=N},E'_{x=N}\ra_Y
\end{align*}
Since $\Pi E$, $\Pi E'$ satisfy the matching conditions at $x=N$, the $\Pi$ part of the latter scalar products vanishes, and only the $\Piperp$ part remains. Writing $\Piperp E=\sum_{k=1}^\infty e^{-\sqrt{\nu_k-\lambda}Lx}\phi_k$, $\phi_k\in V_k$, and similarly for $\Piperp E'$, we get
\begin{align*}
(\lambda-\lambda')\la E,E'\ra_{X^N} & = \sum_{k=1}^\infty (\sqrt{\nu_k-\lambda}-\sqrt{\nu_k-\lambda'}) \la e^{-\sqrt{\nu_k-\lambda}LN}\phi_k,e^{-\sqrt{\nu_k-\lambda'}LN}\phi_k'\ra \\
& = (\lambda'-\lambda) \sum_{k=1}^\infty (\sqrt{\nu_k-\lambda}+\sqrt{\nu_k-\lambda'})^{-1} \la e^{-\sqrt{\nu_k-\lambda}LN}\phi_k,e^{-\sqrt{\nu_k-\lambda'}LN}\phi_k'\ra,
\end{align*}
and since the latter sum is bounded by $$Ce^{-cN} \sqrt{\sum_k \|\phi_k\|^2}\sqrt{\sum_k \|\phi_k'\|^2} = Ce^{-cN} \|\Piperp E_{x=0}\|_Y\|\Piperp E'_{x=0}\|_Y$$ and since $\|E_{x=0}\|\leq C \|E\|_{X^0}$ by elliptic regularity, the claim follows.
\end{proof}

\section{Proof of the Main Theorem}
\label{sec main proofs}
In this section we prove the main theorems. We treat separately the following cases: Eigenvalues on $\XNG$ arising from $L^2-$eigenvalues on $\Xinfty$ below the essential spectrum; eigenvalues on $\XNG$ arising from the continuous spectrum on $\Xinfty$ but away from the threshold, i.e.\ bigger than $\nu+e^{-cN}$ for suitable $c>0$; eigenvalues on $\XNG$ arising from the threshold.

In each case, the proof proceeds in two steps: In Step 1, we construct the eigenvalues on $\XNG$ from approximate eigenfunctions constructed from the (generalized) eigenfunctions on $\Xinfty$. In Step 2 we show that in this way all eigenfunctions are obtained.

We assume at first there $-\DeltaXinfty$ has no $L^2$-eigenvalues in $[\nu,\lambdamax]$. The modifications needed in case there are such eigenvalues are described in Section \ref{subsec L2 eigenvalues}.

We use the following notation: For a selfadjoint operator $A$ and $I\subset\R$ let
$\Eig_I(A)$ be the spectral subspace for the spectral interval $I$. If $A$ has only discrete spectrum in $I$, this is the span of the eigenfunctions of $A$ with eigenvalues in $I$. Also, let
$$ \Eig_{I,N} := \Eig_I(-\DeltaXNG).$$
We will  construct approximate spectral subspaces $\Eigapp_{I,N}$, defined in each case separately, using a cutoff function defined as follows.
Choose $\chi\in C^\infty(\R)$ satisfying $\chi(x)=1$ for $x\leq -\frac12$ and $\chi(x)=0$ for $x\geq -\frac13$ and set
\begin{equation} \label{chiNdef}
\chi_{N}(x)=\chi(x-N).
\end{equation}
Thus, if $u$ is a function on $X^N$ then $\chi_N u$ equals $u$ for $x\leq N-1$ and is identically zero for $x\geq N-\frac13$.

Note that one could not hope to approximate individual eigenfunctions on $\XNG$ by approximate eigenfunctions if the eigenvalues lie very close together.

\subsection{Eigenvalues arising from $L^2$ eigenvalues on $\Xinfty$}
\label{subsec ef L2}

For $I\subset (0,\nu)$ let
\begin{equation}
\label{eqn def Eigapp lambda<nu}
\Eigapp_{I,N} := \Span\{\chi_N u:\, u\in \Eig_I(-\DeltaXinfty)\}.
\end{equation}
\begin{theorem}
\label{thm ev < nu}
For any $c_0>0$ there is $N_0>0$ such that for $N\geq N_0$ the eigenvalues $\mu\leq\nu-c_0$ of $-\DeltaXNG$ lie within $Ce^{-cN}$ of the $L^2$-eigenvalues of $-\DeltaXinfty$.

For each $L^2$-eigenvalue $\lambda<\nu$ of $-\DeltaXinfty$ we have
\begin{equation}
\label{eqn Eig close lambda<nu}
\distsymm(\Eig_{I,N},\Eigapp_{\{\lambda\},N}) \leq Ce^{-cN},\quad I=(\lambda-Ce^{cN},\lambda+Ce^{-cN}).
\end{equation}
Here $c=\sqrt{\nu-\lambdamax}$ where $\lambdamax$ is the largest eigenvalue of $-\DeltaXinfty$ less than $\nu$.
\end{theorem}

\begin{proof}
\noindent{\bf Step 1:}
Show that the approximate eigenfunctions are actually such:
\begin{equation}
\label{eqn Eigapp close to Eig lambda<nu}
\dist(\Eigapp_{\{\lambda\},N},\Eig_{I,N}) <1.
\end{equation}
In particular, $-\DeltaXNG$ has at least $\dim \Eigapp_{\{\lambda\},N}$ many eigenvalues in $I$.

Proof: Let $W=\Eigapp_{\{\lambda\},N}$.
For $w=\chi_{N}u\in W $ one has
$$\Delta w = \chi_{N}(\Delta u) + 2\grad\chi_{N}\grad u + (\Delta \chi_{N})u
= -\lambda \chi_{N}u + 2L^{-2}\chi_{N}'\partial_{x}u + L^{-2}\chi_{N}'' u,$$
and since $\chi_{N}',\chi_{N}''$ are supported in $(N-1,N)$, one obtains from \eqref{expdecaymu<nu}, applied with $M\in (N-1,N)$, that
\begin{equation} \label{eqnwexpest}
\| (\Delta_{X^N_G}+\lambda)w \|_{X^N_G} \leq C e^{-cN} \|w\|_{X^N_G}.
\end{equation}
Since $W\subset \Dom (\Delta_{X^N_G})$ we may apply the Spectral Approximation Lemma \ref{lemmaSAL} to the operator $A=-\Delta_{X^N_G}$, with $\lambda_0=\lambda$ and $\epsilon=C e^{-cN}, \delta=2\epsilon$, and this gives \eqref{eqn Eigapp close to Eig lambda<nu}.
\medskip

\noindent{\bf Step 2:}
Show that any eigenvalue $\mu$ of $-\DeltaXNG$ is in some $I=(\lambda-Ce^{cN},\lambda+Ce^{-cN})$ and that
\begin{equation}
\label{eqn Eig close to Eigapp lambda<nu}
\dist(\Eig_{I,N},\Eigapp_{\{\lambda\},N}) \leq Ce^{-cN}.
\end{equation}
Proof:
Let $u$ be an eigenfunction of $-\Delta_{X^N_G}$, with eigenvalue $\mu\leq\nu-\gamma$.
Then $w=\chi_{N}u \in \Dom (\Delta_{X^\infty})$ satisfies (with $c_1=\sqrt{\gamma}$)
$$ \| (\Delta_{X^\infty} + \mu) w \|_{X^\infty} \leq C e^{-c_1N} \|w\|_{X^\infty}.$$
This follows from exponential decay of $w$ and is proved in the same way as \eqref{eqnwexpest}.
This implies $|\mu-\lambda|\leq C e^{-c_1 N}$ for some $\lambda\in\spec(-\DeltaXinfty)$ (in particular, $c_1$ may be replaced by $c$). Since $\mu\leq\nu-\gamma$, $\lambda$ must be an $L^2$ eigenvalue of $-\DeltaXinfty$. Now apply the Spectral Approximation Lemma \ref{lemmaSAL} to $A=-\DeltaXinfty$, $W=\Span\{w\}$, with $\lambda_0=\mu$ and $\epsilon=C e^{-cN}$, $\delta= \dist(\mu,\spec(-\DeltaXinfty)\setminus\{\lambda\})$. Since the interval $(\mu-\delta,\mu+\delta)$ intersects the spectrum of $-\DeltaXinfty$ only in $\lambda$, we get $\dist(\Span\{w\},\Eig_{\{\lambda\}}(-\DeltaXinfty))<C e^{-cN}$, and this implies (using exponential decay again)
\begin{equation}
\label{eqn u close to Eigapp lambda<nu}
\dist(\Span\{u\},\Eigapp_{\{\lambda\},N}) \leq Ce^{-cN}.
\end{equation}
Finally, applying this to an orthonormal basis of eigenfunctions of $-\DeltaXNG$ with eigenvalues in $I$, we get \eqref{eqn Eig close to Eigapp lambda<nu} from \eqref{eqn dist orthogonal}.

\noindent{\bf End of proof:} The first statement of the theorem is contained in Step 2, and \eqref{eqn Eigapp close to Eig lambda<nu} and \eqref{eqn Eig close to Eigapp lambda<nu} together imply \eqref{eqn Eig close lambda<nu} by Lemma \ref{lemma dist}c).
\end{proof}

\subsection{Eigenvalues arising from the interior of the continuous spectrum}
\label{subsec cont spec}

To define the approximate eigenspaces, recall Theorem \ref{thmscattef1}.
Let $$A_N = \{\alpha_{i,N}^\rho:\, i=1,2,\dots;\rho=1,\dots,r_i,\ \alpha_{i,N}^\rho\in(0,\alphamax]\},\quad \Lambda_N = \{\nu+ \alpha^2:\,\alpha\in A_N\}.$$
Recall that $\alpha\in A_N$ iff $\Pi E_{\alpha,\phi}$ satisfies the matching conditions for some $\phi\neq 0.$
The corresponding space of $\phi$ is $\calP_{\alpha,N}$ from \eqref{phisol}.
Therefore, the function
\begin{equation}
\label{eqn def Etilde}
\Etilde_{\alpha,\phi} = \begin{cases}
                            E_{\alpha,\phi} & \text{ on } X^0\\
                            \Pi E_{\alpha,\phi} + \chi_N \Piperp E_{\alpha,\phi} & \text{ on } Z^N
\end{cases}
\end{equation}
is smooth on $\XNG$ and in the domain of $\DeltaXNG$.

For $I\subset\R$ let
$$\Eigapp_{I,N} := \Span\{\Etilde_{\alpha,\phi},\ \nu+\alpha^2\in I,\ \alpha\in A_N,\ \phi\in \calP_{\alpha,N}\}.$$
Also, if $I=(A,B)$ and $\delta>0$ then let $$I_\delta := (A-\delta,B+\delta).$$

We call  $\dim\calP_{\alpha,N}$ the multiplicity of $\alpha$ resp. of $\lambda=\nu+\alpha^2$.

\begin{theorem}
\label{thm main general theorem lambda>nu}
Assume $-\DeltaXinfty$ has no $L^2$-eigenvalues in $[\nu,\lambdamax]$. The numbers $\lambda\in \Lambda_N$ are the approximate eigenvalues of $-\DeltaXNG$, with approximate eigenfunctions linear combinations of $\Etilde_{\alpha,\phi}$, $\phi\in\calP_{\alpha,N}$. The errors are of order $e^{-cN}$.

More precisely, given a sufficiently small $c_0>0$ there are constants $C,c>0$ such that:
\begin{enuma}
\item
Let $\mu_i$ be the eigenvalues bigger than $\nu+e^{-2c_0N}$ of $-\DeltaXNG$, arranged in increasing order and counted with multiplicity. Also, let $\lambda_i$ be the elements of $\Lambda_N$, arranged in increasing order and counted with multiplicity. Then, for all $i$,
\begin{equation}
\label{eqn evdistance}
|\lambda_i-\mu_i| \leq Ce^{-cN}.
\end{equation}
\item
Let $I\subset (\nu+Ce^{-cN},\lambdamax]$. If there is no $\lambda\in \Lambda_N$ in $I_{2Ce^{-cN}}\setminus I$ then
\begin{equation}
\label{eqn Eigclose lambda>nu}
\distsymm (\Eig_{I',N},\Eigapp_{I,N}) \leq  C e^{-cN},\quad I' = I_{Ce^{-cN}} .
\end{equation}
\end{enuma}
\end{theorem}
See Section \ref{subsec L2 eigenvalues} for the modifications needed in case there are $L^2$-eigenvalues in $[\nu,\lambdamax]$.

The statement in b) is complicated due to the possible crossings of the branches $z_i^\rho$ for different $i$ in Figure \ref{figure2}. These do not occur on the line $z=\alpha N$ for bounded $i$ (corresponding to fixed $k$ as in Theorem \ref{maintheorem_Gepsilon}), and one obtains:
\begin{corollary}
\label{cor bounded i}
The eigenvalues of $-\DeltaXNG$ form clusters of width $Ce^{-cN}$ around the $\lambda\in \Lambda_N$. For any $C_0>0$ there are $c>0,N_0$ such that for $N\geq N_0$ the clusters around the $\lambda\leq \nu+C_0 N^{-2}$ are disjoint and the span of eigenfunctions of $-\DeltaXNG$ corresponding to the $\lambda$-cluster has distance less than $e^{-cN}$ from
$$\Span\{\Etilde_{\alpha,\phi}: \phi\in \calP(\alpha)\}$$
where $\lambda=\nu+\alpha^2$, $\calP(\alpha)= \ker (I-e^{i\alpha N2L}\sigma S(\alpha))$.
\end{corollary}
\begin{proof}
The first statement is just \eqref{eqn evdistance}. There are only a bounded number of $\lambda\in\Lambda_N$ satisfying $\lambda\leq \nu+C_0N^{-2}$ by the Weyl asymptotics \eqref{eqn asymptotics}. They are polynomially separated, i.e. there are $K\in\N, c>0$ (depending on $C_0$) so that $|\lambda-\lambda'|\geq cN^{-K}$ for any two different such values, by Theorem \ref{thmscattef1} and the fact that the $Z_i^\rho$ are different analytic functions. This implies the disjointness of clusters for large $N$  and that the separation condition in Theorem \ref{thm main general theorem lambda>nu}b) is satisfied for $I=\{\lambda\}$, and this gives the last claim.
\end{proof}
\begin{proof}[Proof of Theorem \ref{thm main general theorem lambda>nu}]
As always, we write $\lambda=\nu+\alpha^2$.

\noindent{\bf Step 1:} Show that the approximate eigenfunctions are actually such: For any intervals as in the theorem we have, for $C$ sufficiently large,
\begin{equation}
\label{eqn Eigapp close to Eig lambda>nu}
\dist (\Eigapp_{I,N},\Eig_{I',N}) < 1.
\end{equation}

Proof: Let $\lambda\in\Lambda_N$ and let $W_\lambda=\Eigapp_{\{\lambda\},N}$.
For $w=\Etilde_{\alpha,\phi}\in W_\lambda$ one has
$$\Delta w = -\lambda w + 2L^{-2}\chi_{N}'\partial_{x}\Piperp E_{\alpha,\phi} + L^{-2}\chi_{N}'' \Piperp E_{\alpha,\phi},$$
and since $\chi_{N}',\chi_{N}''$ are supported in $(N-1,N)$, one obtains from \eqref{expdecaymu>nu}, applied to $u=E_{\alpha,\phi}$ with $M\in (N-1,N)$, that
\begin{equation} \label{eqnwexpest2}
\| (\Delta_{X^N_G}+\lambda)w \|_{X^N_G} \leq C e^{-cN} \|w\|_{X^N_G}.
\end{equation}
Since $W_\lambda\subset \Dom (\Delta_{X^N_G})$ we may apply the Spectral Approximation Lemma \ref{lemmaSAL} to the operator $A=-\Delta_{X^N_G}$, with $\lambda_0=\lambda$ and $\epsilon=C e^{-cN}$ and $\delta>0$ to be chosen, and this gives
\begin{equation}
\label{eqn Eigapp close to Eig lambda>nu 1}
\dist(\Eigapp_{\{\lambda\},N},\Eig_{(\lambda-\delta,\lambda+\delta),N}) \leq C e^{-cN}/\delta.
\end{equation}
If $\delta=2C e^{-cN}$ then this is \eqref{eqn Eigapp close to Eig lambda>nu} with $I=\{\lambda\}$.

To obtain \eqref{eqn Eigapp close to Eig lambda>nu} for arbitrary $I$,  first observe that \eqref{eqn Eigapp close to Eig lambda>nu 1} implies $\dist(\Eigapp_{\{\lambda\},N},\Eig_{J,N}) \leq C e^{-cN}/\delta$ for any interval $J$ containing $(\lambda-\delta,\lambda+\delta)$. Next,
use Lemma \ref{lemma E almost orthogonal} together with a version of \eqref{eqn dist orthogonal} for almost orthogonal subspaces to conclude that
$\dist (\Eigapp_{I,N},\Eig_{I_\delta,N}) \leq \sum\limits_{\lambda\in I\cap\Lambda_N} (\dist(\Eigapp_{\{\lambda\},N},I_\delta) + Ce^{-cN})$. Since $\Lambda_N$ has at most $O(N)$ elements,
the left hand side is bounded by $NCe^{-cN}/\delta$. Hence, choosing $\delta=2NC e^{-cN}$ (or $e^{-c'N}$ with $c'$ smaller than $c$ and $N$ large) one obtains \eqref{eqn Eigapp close to Eig lambda>nu}.

\medskip

\noindent{\bf Step 2:} Show that each eigenvalue $\mu>\nu+e^{-2c_0N}$ of $-\DeltaXNG$ is exponentially close to some $\lambda\in \Lambda_N$ and that, under the assumptions of the theorem,
\begin{equation}
\label{eqn Eig close to Eigapp lambda>nu}
\dist (\Eig_{I',N},\Eigapp_{I,N}) \leq C e^{-cN}.
\end{equation}

Proof: Let $(\DeltaXNG+\mu)u=0$, $\mu=\nu+\beta^2$, so $\beta>e^{-c_0N}$.
For $E=E_{\alpha,\phi}$ recall the notation $(E^0,E^1)=(\Pi E_{x=0},\Pi\partial_\xi E_{x=0})\in V\times V$.
For $\alpha\in A_N$ denote
\begin{equation}
\label{eqn def calF_alpha}
\calF_{\alpha,N} = \{(E^0,E^1):\, E=E_{\alpha,\phi},\ \phi\in\calP_{\alpha,N}\}.
\end{equation}
\noindent{\bf Step 2a:}
$(u^0,u^1)$ is close to $\calL_\beta$ by the elliptic estimate:
\begin{equation}
\label{eqn u0u1 close to calL}
\dist(\Span\{(u^0,u^1)\}, \calL_\beta) \leq C e^{-cN}.
\end{equation}

Proof: Apply the basic elliptic estimate, Lemma \ref{lemma basic ell estimate}, as follows: Let $(E^0,E^1)$ be the orthogonal projection of $(u^0,u^1)$ to $\calL_\beta$. This corresponds to a scattering solution $E$ for spectral value $\mu$. Let $v=u-E$. From $(\Delta+\mu)E=0$, $\Bperp(\mu)E=0$ it follows that $v$ satisfies \eqref{eqn ell system} with $f=0$, $g=\Bperp(\mu)u$, $\calL_\beta'=$ the orthogonal complement of $\calL_\beta$ and $h=0$. From Lemma \ref{lemma eigenfcn est} it follows that $\|g\|_{H^{1/2}(Y)}\leq C e^{-cN}\|u\|_{H^2(X^0)}$, and then \eqref{eqn basic ell estimate} gives
\begin{equation}
\label{eqn v small in H2}
\|v\|_{H^2(X^0)} \leq Ce^{-cN} \|u\|_{H^2(X^0)},\quad v=u-E.
\end{equation}
This implies $\|u\|_{H^2(X^0)} \leq C\|E\|_{H^2(X^0)}$, and with \eqref{eqn E bounded by E0E1} we get
\begin{equation}
\label{eqn u u0u1 estimate}
\|u\|_{H^2(X^0)} \leq C\|(E^0,E^1)\|_{V\times V}.
\end{equation}
Next, the trace theorem implies $\|(v^0,v^1)\|_{V\times V} \leq C\|v\|_{H^2(X^0)}$, and so \eqref{eqn v small in H2} gives
\begin{equation}
\label{eqn v0v1 small}
\|(v^0,v^1)\|_{V\times V}\leq Ce^{-cN}\|(E^0,E^1)\|_{V\times V},
\end{equation}
which after writing $E=u-v$ and absorbing the $v$ term into the left hand side gives
$\|(v^0,v^1)\|_{V\times V}\leq Ce^{-cN}\|(u^0,u^1)\|_{V\times V},$ that is \eqref{eqn u0u1 close to calL}.
\medskip

\noindent{\bf Step 2b:}
Use the matching conditions to show: \eqref{eqn u0u1 close to calL} implies
\begin{align}
\label{eqn beta close to AN}
\dist(\beta,A_N) &\leq  C e^{-cN}\\
\label{eqn u0u1 close to sum of F}
\dist(\Span\{(u^0,u^1)\},\bigoplus_{\alpha\in A_N:|\alpha-\beta|<Ce^{-cN}} \calF_{\alpha,N}) & \leq C e^{-cN}.
\end{align}
The $\calF_{\alpha,N}$ are defined in \eqref{eqn def calF_alpha}.

Proof: Since $(\Delta_{X^N}+\mu)v=0$ one obtains from \eqref{eqn Pi u from u0u1}, using $|\cos{\beta N}|\leq 1,\ |\beta^{-1}\sin(\beta N)|\leq N$, that
$$ \|\BC_\beta(v)\|_V \leq \|(v^0\|_V+(N+\frac1\alpha)\|v^1)\|_{V}\leq e^{c_0N}\|(v^0,v^1\|_{V\times V}.$$
From $\BC_\beta(u)=0$ we have $\BC_\beta(E)=-\BC_\beta(v)$, and then \eqref{eqn v0v1 small}  gives
\begin{equation}
\label{eqn MC estimate for E}
\|\BC_\beta(E)\|_V = \|\BC_\beta(v)\|_V \leq Ce^{-c_1N} \|(E^0,E^1)\|_{V\times V}.
\end{equation}
where $c_1=c-c_0$.
If $E=E_{\beta,\phi}$ then clearly $\|(E^0,E^1)\|_{V\times V}\leq C\|\phi\|_V$, so
we can apply the Stability Theorem \ref{thm stability} with $\delta=Ce^{-c_1N}$ and obtain \eqref{eqn beta close to AN} with $c=c_1$, as well as
$\dist(\Span\{\phi\},\bigoplus_{\alpha'}\calP_{\alpha',N})\leq Ce^{-c_2N}$ where $c_2=c_1/(2+2\dim V)$ and the sum is over $\alpha'\in A_N$ satisfying $|\alpha-\alpha'|\leq Ce^{-c_1N/2}$. This implies that $(E^0,E^1)$ and hence, by \eqref{eqn v0v1 small}, $(u^0,u^1)$ has distance at most $Ce^{-c_2N}$ from $\bigoplus_{\alpha'} \calF_{\alpha',N}$, and hence \eqref{eqn u0u1 close to sum of F}, with $c=c_2$.

End of proof of Step 2: The estimate
$\|v\|_{X^N}\leq CN\|(v^0,v^1)\|_{V\times V}$ for eigensolutions on $X^N$ which are a difference of a scattering solution and an eigenfunction on $\XNG$ (use \eqref{expdecaymu>nu} and a modification of the derivation of \eqref{eqn E bounded by E0E1}) shows that \eqref{eqn u0u1 close to sum of F} implies
\begin{equation}
\label{eqn u close to Eigapp lambda>nu}
\dist(\{u\}, \Eigapp_{J,N}) \leq Ce^{-cN},\quad J=(\mu-Ce^{-cN},\mu+Ce^{-cN}).
\end{equation}
%Here, we used that $\alpha>e^{-c_0 N}$, and $c$ is smaller than the previous constants $c$.
Finally, we apply this to an orthonormal set of eigenfunctions $u$ with eigenvalues in $I'$. Lemma \ref{lemma dist}e) gives $\dist(\Eig_{I',N},\Eigapp_{I'',N})\leq Ce^{-cN}$ with $I''=(I')_{Ce^{-cN}}=I_{2Ce^{-cN}}$ and hence
\eqref{eqn Eig close to Eigapp lambda>nu}, since by assumption any $\lambda\in I''\cap\Lambda_N$ must already lie in $I$.

\medskip

{\bf End of proof of Theorem \ref{thm main general theorem lambda>nu}:}
\eqref{eqn Eigapp close to Eig lambda>nu} and \eqref{eqn Eig close to Eigapp lambda>nu} give part b) of the Theorem by Lemma \ref{lemma dist}b). Part a) then follows easily: Since $\Lambda_N$ has $O(N)$ elements, we may cover it by intervals $I_k$ of length at most $CNe^{-cN}$, satisfying the hypothesis of b) (note that any $\lambda\in\Lambda_N$ is at least $\nu+cN^{-2}$ by \eqref{eqn initial},\eqref{alphasol} since $z_1>0$). The $\mu_i$ must then be in the $Ce^{-cN}$-neighborhoods of the $I_k$ by b), and this implies a), with slightly smaller $c$.
\end{proof}

\subsection{Eigenvalues close to the threshold $\nu$}
Recall from Lemma \ref{lemmabcEthreshold} that $\Pi E_{0,\Phi_+,\Psi_-}$ satisfies the matching conditions at $x=N$ iff $\sigma\Phi_+=\Phi_+$ and $\Psi_-=0$. Denote
$$ V_+^+ := V_+\cap V^+, \quad \Phi_+^+=\text{ the projection to }V_+^+\text{ of }\Phi\in V.$$
Note that, since $\sigma$ and $S(0)$ do not commute in general, some care is needed with this notation. For example, usually $\Phi_+\neq \Phi_+^++\Phi_+^-$ (with $\Phi_+^-$ defined analogously).

For $\Phi\in V_+^+$ let
\begin{equation}
\label{eqn def Etilde0}
\Etilde_{0,\Phi,0} := \begin{cases}
                            E_{0,\Phi,0} & \text{ on } X^0\\
                            \Pi E_{0,\Phi,0} + \chi_N \Piperp E_{0,\Phi,0} & \text{ on } Z^N
\end{cases}
\end{equation}
and
\begin{equation}
\label{eqn def Eigapp nu}
\Eigapp_{\nu,N} := \{\Etilde_{0,\Phi,0}:\,\Phi\in V^+_+\}.
\end{equation}

Here we prove the following:
\begin{theorem} \label{thm general theorem lambda close to nu}
Suppose $\nu$ is not an $L^2$-eigenvalue of $-\DeltaXinfty$. There are $c_0>0$ and $c,C>0$ such that all eigenvalues of $-\DeltaXNG$ in the interval $(\nu-c_0,\nu+e^{-2c_0 N})$ are actually in
$I:=(\nu-e^{-cN},\nu+e^{-cN})$, and
\begin{equation}
\label{eqn Eigclose lambda=nu}
\distsymm(\Eig_{I,N},\Eigapp_{\nu,N}) \leq Ce^{-cN}.
\end{equation}
\end{theorem}
See Section \ref{subsec L2 eigenvalues}, esp.\ \eqref{eqn def Eigapp nu with eigenvalue}, for the modification needed in case $\nu$ is an $L^2$-eigenvalue.
\begin{proof}
{\bf Step 1:} Show that the approximate eigenvalues are actually such:
\begin{equation}
\label{eqn Eigapp close to Eig lambda=nu}
\distsymm(\Eigapp_{\nu,N},\Eig_{I,N}) < 1.
\end{equation}
Proof: This is proved in exactly the same way as \eqref{eqn Eigapp close to Eig lambda>nu} (with $I=\{\lambda\}$).
\medskip

\noindent{\bf Step 2:} Show that each eigenvalue $\mu\in (\nu-c_0,\nu+e^{-2c_0N})$ of $-\DeltaXNG$ is
in $I$ and that
\begin{equation}
\label{eqn Eig close to Eigapp lambda=nu}
\dist (\Eig_{I,N},\Eigapp_{\nu,N}) \leq C e^{-cN}.
\end{equation}

Proof: Define
\begin{equation}
\label{eqn def calF_0}
\calF_{0,N} = \{(E^0,E^1):\, E=E_{0,\Phi,0},\ \Phi\in V_+^+\}
\end{equation}
analogous to \eqref{eqn def calF_alpha}.
Let $u$ be an eigenfunction of $-\DeltaXNG$, with eigenvalue $\mu$.
Let $$\delta=|\nu-\mu|.$$
Since  there are no scattering solutions with $\mu<\nu$ we compare $u$ with a scattering solution for $\lambda=\nu$.

\noindent{\bf Step 2a:}
$(u^0,u^1)$ is close to $\calL_0$ by the elliptic estimate:
\begin{equation}
\label{eqn PhiPsi close to calL0}
\dist(\Span\{(u^0,u^1)\}, \calL_0) \leq C\delta',\quad \delta' := \delta +  e^{-cN}.
\end{equation}

Proof: Denote
$$\Phi = u^0,\ \Psi=u^1,\quad E=E_{0,\Phi_+,\Psi_-}.$$
We apply the elliptic estimate, Lemma \ref{lemma basic ell estimate}, to the difference $v=u-E$, with $\lambda=\nu$ and  $\calL'_0=\{(v^0,v^1): \, (v^0)_+=0,(v^1)_-=0\}$. By \eqref{eqn L0}, this is transversal to $\calL_0$.
$v$ satisfies
$(\Delta + \nu) v = (\Delta+\nu)u=(\nu-\mu) u$ and $P_{\calL_0,\calL_0'}(v^0,v^1)=0$ (since $(v^0,v^1)\in \calL_0'$ by construction).
Also, $\Bperp(\nu)v = \Bperp(\nu)u= \Bperp(\mu)u + (\Bperp(\nu)-\Bperp(\mu))u$, and \eqref{eqn Bperp estimate} gives
$\|\Bperp(\mu)u\|_{H^{1/2}(Y)} \leq Ce^{-cN}\|u\|_{H^2(X^0)}$ while clearly $\|(\Bperp(\nu)-\Bperp(\mu))u\|_{H^{1/2}(Y)}\leq C|\nu-\mu|\cdot\|u\|_{H^2(X^0)}$.
The elliptic estimate \eqref{eqn basic ell estimate} now gives
\begin{equation}
\label{eqn vu est l=n}
\|v\|_{H^2(X^0)}\leq C\delta' \|u\|_{H^2(X^0)},\quad \delta' := \delta + e^{-cN}.
\end{equation}
For $\delta$ sufficiently small, this implies $\|u\|_{H^2(X^0)}\leq C\|E\|_{H^2(X^0)}$. Using \eqref{eqn E bounded by E0E1}  and $\|(v^0,v^1)\|_{V\times V}\leq C\|v\|_{H^2(X^0)}$ from the trace theorem we get
$\|(v^0,v^1)\|_{V\times V}\leq C\delta' \|(E^0,E^1)\|_{V\times V}$, that is, \eqref{eqn PhiPsi close to calL0}.
\medskip

\noindent{\bf Step 2b:}
Use the matching conditions to show: \eqref{eqn PhiPsi close to calL0} implies that there is a constant $c_0>0$ so that for $\mu>\nu-c_0$
\begin{align}
\label{eqn mu close to nu}
|\mu-\nu| &\leq  Ce^{-cN}\\
\label{eqn u0u1 close to F_0}
\dist(\Span\{(u^0,u^1)\},\calF_{0,N}) & \leq C e^{-cN}.
\end{align}

Proof: First,
note that  $E^0 = \Phi_+ + T\Psi_-$ (with $T:=\frac i2 S'(0)$), $E^1=\Psi_-$ and $v^0=u^0-E^0 = \Phi_- - T\Psi_-$, $v^1=u^1-E^1 = \Psi_+$ imply that \eqref{eqn PhiPsi close to calL0} is equivalent to
\begin{equation}
\label{eqn Phi-Psi+ estimate}
\|\Phi_- - T\Psi_-\| + \|\Psi_+\| \leq C\delta' (\|\Phi_+\| + \|\Psi_-\|).
\end{equation}
Using $\langle\Phi,\Psi\rangle = \langle\Phi_+,\Psi_+\rangle + \langle\Phi_-,\Psi_-\rangle$ one gets from this
\begin{equation}
\label{eqn Phi times Psi estimate}
|\langle\Phi,\Psi\rangle| \leq  C\delta' \|\Phi\|^2 + C\|\Psi\|^2.
\end{equation}
Note that this estimate does not involve the $V_\pm$ splitting. This is essential for the argument.

We now consider the cases $\mu<\nu$ and $\mu\geq\nu$ separately.

\medskip

\noindent{\bfseries The case $\mu<\nu$}

For the sake of clarity we assume for the following argument that $L=I$. The case of general $L$ requires only adjusting the constants.

Let $a=\sqrt\delta$ and $t=\tanh aN$.
The matching conditions \eqref{eqn matching u} for  $\Pi u=\Phi\cosh a x + \Psi \frac {\sinh ax}a$ are
\begin{align}
\label{eqn matching-}
\Phi^- + a^{-1}t\Psi^- &= 0\\
\label{eqn matching+}
at\Phi^+ + \Psi^+ & = 0.
\end{align}
This implies $\langle\Phi,\Psi\rangle = \langle\Phi^+,\Psi^+\rangle + \langle\Phi^-,\Psi^-\rangle = -at\|\Phi^+\|^2 - a^{-1}t\|\Psi^-\|^2$ and so
\begin{align}\label{eqn 1}
|\langle\Phi,\Psi\rangle| & = at \|\Phi^+\|^2 + a^{-1}t\|\Psi^-\|^2 \\
\|\Phi\|^2  & =  \|\Phi^+\|^2 + a^{-2}t^2\|\Psi^-\|^2 \\
\label{eqn 3}
\|\Psi\|^2 & = a^2t^2 \|\Phi^+\|^2 + \|\Psi^-\|^2
\end{align}
Now \eqref{eqn Phi times Psi estimate} implies that at least one of the following inequalities must hold:
\begin{align}
\label{eqn ineq a}
at & \leq C\delta' + Ca^2 t^2\\
\label{eqn ineq b}
a^{-1}t & \leq C\delta' a^{-2} t^2 + C.
\end{align}
Multiply the second inequality by $a^2$, plug in $\delta'=a^2+e^{-cN}$ and use $0<t<1$ to see that the second inequality implies the first. So \eqref{eqn ineq a} holds. We claim that there is $a_0>0$ so that $a<a_0$ implies $a\leq Ce^{-cN}$. To see this, first observe that the $Ca^2t^2$ term on the right may be absorbed into the left for sufficiently small $a$, since $t<1$. So we get $at\leq Ca^2 + Ce^{-cN}$. Now for $a>N^{-1}$ we have $t\geq \tanh 1$, so the $a^2$ term may be absorbed into the left, which yields $a\leq Ce^{-cN}$, while for $a\leq N^{-1}$ we have $t\geq c' aN$ for some constant $c'>0$, and this gives $a\leq Ce^{-cN/2}$.

We have shown that $a^2=|\nu-\mu|\leq Ce^{-cN}$ if $a<a_0$, that is, \eqref{eqn mu close to nu}.

In particular, $a^{-1}t\sim N$. Now use \eqref{eqn Phi times Psi estimate} again in conjunction with \eqref{eqn 1}-\eqref{eqn 3}, where we keep only the $\Psi^-$ term on the left hand side, to obtain
$N\|\Psi^-\|^2 \leq Ce^{-cN}(\|\Phi^+\|^2 + N^2\|\Psi^-\|^2) + C\|\Psi^-\|^2$. For large $N$ the $\Psi^-$ terms on the right may be absorbed, and one obtains
$\|\Psi^-\|\leq Ce^{-cN}\|\Phi^+\|$. Together with \eqref{eqn matching-} this gives
\begin{equation}
\label{eqn Phi^- very small}
\|\Phi^-\| \leq Ce^{-cN}\|\Phi\|.
\end{equation}
Also, since \eqref{eqn matching+} gives $\|\Psi^+\|\leq Ce^{-cN}\|\Phi^+\|$, we obtain
\begin{equation}
\label{eqn Psi very small}
\|\Psi\| \leq e^{-cN}\|\Phi\|.
\end{equation}
Now \eqref{eqn Phi-Psi+ estimate} implies
\begin{equation}
\label{eqn Phi_- very small}
\|\Phi_-\|\leq e^{-cN} \|\Phi\|.
\end{equation}
Finally, \eqref{eqn Phi^- very small} and \eqref{eqn Phi_- very small}
imply by an elementary argument
%\begin{equation}
%\label{eqn Phi++ bound}
%\end{equation}
$\|\Phi-\Phi_+^+\| \leq Ce^{-cN} \|\Phi\|$. Therefore, $\|(\Phi,\Psi)-(\Phi_+^+,0)\|\leq Ce^{-cN}\|(\Phi,\Psi)\|$, that is, \eqref{eqn u0u1 close to F_0}.
\medskip

\noindent{\bf The case $\mu\geq\nu$:}
Let again $a=\sqrt{\delta}$, but now $t=\tan aN$, with $a^{-1}t:=N$ if $a=0$. Since we assume $a\leq Ce^{-cN}$ for this case, we may argue as in the last part of the argument for $\mu<\nu$ (starting with the paragraph before \eqref{eqn Phi^- very small}). Observe that now \eqref{eqn matching+} is replaced by $-at\Phi^++\Psi^+=0$, which yields
$|\langle\Phi,\Psi\rangle| = |at\|\Phi^+\|^2 - a^{-1}t\|\Psi^-\|^2|$ instead of \eqref{eqn 1} (so \eqref{eqn Phi times Psi estimate} gives only weaker conclusions than before), but the conclusions are still valid since $at\|\Phi^+\|^2\leq Ce^{-cN}\|\Phi^+\|^2$.
\medskip

End of proof of Step 2: Exactly as in the proof of Theorem \ref{thm main general theorem lambda>nu}, it follows from Step 2b that $\dist(\Span\{u\},\Eigapp_{0,N})\leq Ce^{-cN}$; applying this to an orthonormal basis of $\Eig_{I,N}$ and using Lemma \ref{lemma dist}d),e) we get the claim.
\medskip

\noindent{\bf End of proof of Theorem \ref{thm general theorem lambda close to nu}:}
The claim follows directly from Steps 1 and 2, using Lemma \ref{lemma dist}b).
\end{proof}
\subsection{The case of embedded $L^2$-eigenvalues}
\label{subsec L2 eigenvalues}
Here we sketch the modifications necessary in the arguments to deal with the case that $-\DeltaXinfty$ has $L^2$-eigenvalues embedded in the essential spectrum. For simplicity we will restrict to the analysis of eigenvalues near $\nu$, in case that $\nu$ is an eigenvalue of $-\DeltaXinfty$. The case of $L^2$-eigenvalues bigger than $\nu$ is treated similarly.

Let $\calH= \{u\in L^2(\Xinfty):\, (\DeltaXinfty + \nu)u=0\}.$
If $\calH\neq 0$, Theorem \ref{thm general theorem lambda close to nu} holds with the definition of $\Eigapp_{\nu,N}$ replaced by
\begin{equation}
\label{eqn def Eigapp nu with eigenvalue}
\Eigapp_{\nu,N}:= \{\Etilde_{0,\Phi,0}:, \Phi\in V_+^+\} + \{\chi_N u:\, u\in\calH\}.
\end{equation}
Also, Theorem \ref{thm main general theorem lambda>nu} continues to hold as stated (if there are embedded eigenvalues $\lambda>\nu$ then its statement has to be modified in a straightforward way).

To prove this, we have to first modify the elliptic estimate, Lemma \ref{lemma basic ell estimate}. We are interested in  $\alpha$ near $0$. Let $\calH_0$ be the space of restrictions of elements of $\calH$ to $X^0$ and $\calH_0^\perp$ its orthogonal complement in $L^2(X^0)$. Then the elliptic estimate as stated cannot hold since the homogeneous problem (i.e., $f=g=h=0$ in \eqref{eqn ell system}) has solution space $\calH_0$. However, the same argument as given there shows that the same estimate holds if $u\in\calH_0^\perp$, and this gives
\begin{equation}
\label{eqn basic ell estimate modified}
\|u-P_0u\|_{H^2(X^0)} \leq C (\; \|f\|_{L^2(X^0)} + \|g\|_{H^{1/2}(Y)} + \|h\|_{V\times V}\;)
\end{equation}
where $P_0:L^2(X^0)\to\calH_0$ denotes the orthogonal projection.

Next, we have the following almost orthogonality statement analogous to Lemma \ref{lemma E almost orthogonal}:

If $u$ is an eigenfunction of $-\DeltaXNG$, with eigenvalue $\mu\neq\nu$, then
\begin{equation}
\label{eqn u u_0 almost orthogonal}
\|P_0u\|_{H^2(X^0)} \leq C\frac{e^{-cN}}{|\mu-\nu|} \|u\|_{X^0}.
\end{equation}
For the proof it suffices to show the same estimate for the $L^2(X^0)$ norm of $P_0u$ by standard elliptic regularity, and for this we need to show $|\langle u,u'\rangle| \leq C\frac{e^{-cN}}{|\mu-\nu|}\|u\|\cdot\|u'\|$ for all $u'\in\calH_0$, with scalar product and norms in $L^2(X^0)$. For this, do the analogous calculation as at the start of the proof of Lemma \ref{lemma E almost orthogonal}, then use that $\Pi u'=0$ and that $\|u'_{x=N}\|_Y \leq Ce^{-cN}\|u'\|$, $\|\Piperp u_{x=N}\|_Y\leq Ce^{-cN}\|u\|$, with the same estimate for the $\xi$-derivatives.

Now the proof of Theorem \ref{thm main general theorem lambda>nu} goes through as before since for $\mu>\nu+Ce^{-c_0N}$ \eqref{eqn u u_0 almost orthogonal} shows that \eqref{eqn basic ell estimate modified} reduces to the 'old' elliptic estimate \eqref{eqn basic ell estimate}.

For the proof of Theorem \ref{thm general theorem lambda close to nu} we first observe that Step 1 may be proved simply by a combination of the proofs of the Steps 1 in Theorems \ref{thm ev < nu} and \ref{thm main general theorem lambda>nu}. Next, for Step 2 we may assume right away that $|\mu-\nu|\leq Ce^{-cN}$ since otherwise the 'old' elliptic estimate holds (see the previous paragraph) and the proof does not need to be modified.
Now for an eigenfunction $u$ of $-\DeltaXNG$ let $u_0\in\calH$ be the eigenfunction of $-\DeltaXinfty$ restricting to $P_0u$, and let $\utilde=u-u_0$. Then, since $\utilde_{|X^0}=(I-P_0)(u_{|X^0})$, \eqref{eqn basic ell estimate modified} is just the 'old' elliptic estimate for $\utilde$, so the proofs of Steps 2a and 2b go through for $\utilde$ instead of $u$ as before (with minor modifications when using \eqref{eqn Bperp estimate}, and the matching conditions only satisfied up to an exponentially small error because of the $u_0$ term, which is inessential for the resulting estimate), and this gives that $\utilde$ is exponentially close to an $\Etilde_{0,\Phi,0}$ and hence that $u=u_0+\utilde$ is exponentially close to $\Eigapp_{\nu,N}$.

\subsection{Proof of Theorem \ref{maintheorem1}}
First, choose $c_0>0$ so that the conclusion of Theorem \ref{thm general theorem lambda close to nu} holds. Since $\calP_0=V_+^+$, \eqref{eqn Eigclose lambda=nu} gives the eigenvalues in Theorem \ref{maintheorem1}b), by Lemma \ref{lemma dist}b)
(and actually precise information on the eigenfunctions). Next, with this $c_0$ apply Theorem \ref{thm ev < nu}, then \eqref{eqn Eig close lambda<nu} gives the eigenvalues a) (for $\tau_p<\nu$), and Theorem \ref{thm main general theorem lambda>nu}, then \eqref{eqn evdistance} gives the eigenvalues in c). The eigenvalues close to those $\tau_p$ which are $\geq\nu$ are obtained using the argument in the preceding subsection. The cited theorems also give that there are no other eigenvalues.

\section{Identifying the quantum graph; special cases}
\label{sec quantum graph}
\begin{proof}[Proof of Theorem \ref{mainthmquantumgraph}]
We first discuss how to obtain the eigenvalues of a quantum graph. The metric graph $(G,2L)$ (that is, the graph $G$ with given edge lengths $2l_e$, considered as a one-dimensional simplicial complex, i.e. as a union of intervals glued at the vertices) is just the space $X_G^1$ defined in \eqref{eqn def XNG}, with vertex and edge manifolds all equal to a point. Here we disregard the dimension requirement on the vertex and edge manifolds; but since the dimension requirement was never used (except implicitly in the validity of the theorems of scattering theory) we may use all previous results except those on scattering theory.
Scattering theory is replaced as follows. A boundary condition at the vertices of $G$ corresponds to a scattering matrix $S_G(\alpha)$, defined for $\alpha\neq0$ by the requirement that the function $e^{-i\alpha\xi}\phi_G + e^{i\alpha\xi}S_G(\alpha)\phi_G$ on $X^1$ satisfy the boundary condition for each $\phi_G\in V_G$. By Lemma \ref{lemmabcE} this function satisfies the matching condition at $x=1$, i.e.\ extends to a smooth function on the metric graph, iff
\begin{equation}
\label{eqn eveqn on G}
(I-e^{i\alpha 2L} \sigma S_G(\alpha) )\phi_G = 0.
\end{equation}
Since $\nu_G=0$, this means that the positive eigenvalues of this quantum graph are the squares of those $\alpha$ for which \eqref{eqn eveqn on G} has a solution $\phi_G\neq0$ (counted with multiplicity, defined as dimension of the space of those $\phi_G$).

On the other hand, from \eqref{eqn alpha-formel} and \eqref{eqn initial} we have that the positive $b_k$ in Theorem \ref{maintheorem_Gepsilon} are precisely the squares of those $z>0$ for which
$(I-e^{iz2L}\sigma S(0))\phi$ has a solution $\phi\neq 0$, counted with multiplicity.

It follows that we should choose boundary conditions for the quantum graph such that
\begin{equation}
\label{eqn SG}
 S_G(\alpha) = S(0) \quad\text{ for all }\alpha.
\end{equation}
In particular, we should take $V_G=V$, which leads us to consider functions which on the edge $e$ take values in $\calN_e$; Lemma \ref{lemma scatt subspace} shows that \eqref{eqn SG} yields the boundary conditions \eqref{eqn bc1}, \eqref{eqn bc2}. Finally, Lemma \ref{lemmabcEthreshold} shows that also the zero eigenvalues of the quantum graph correspond to the $b_k=0$.
\end{proof}
Note that the proof also gives a correspondence of the leading parts of eigenfunctions (since they are determined by $\phi_G$ and $\phi$).

We recover previously known results easily. The operator on the quantum graph in Theorem \ref{mainthmquantumgraph} is sometimes called the limit operator. For the following statement, see the remarks after that theorem, and for the notation the beginning of Section \ref{secestef}.
\begin{theorem}
\label{thm special cases}
Suppose all vertex and edge manifolds are connected.
Let $\lambda_0$ be the smallest eigenvalue of $-\Delta_{X^0,\calN}$, where $\calN$ means that, in addition to the D/N (resp. Robin) boundary conditions at $\partial X^0\setminus Y$, we impose Neumann boundary conditions at $Y$. Then $-\DeltaXinfty$ has no $L^2$-eigenvalues $\leq \lambda_0$, and:
\begin{enuma}
\item
If $\lambda>\nu$ then we have Dirichlet conditions, i.e. decoupling, in the limit operator.
\item
If $\lambda=\nu=0$ then we have Kirchhoff boundary conditions in the limit operator.
\end{enuma}
\end{theorem}
In particular, for Neumann boundary conditions on all of $\partial \XNG$ one has Kirchhoff boundary conditions, as proved in \cite{ExnPos:CSGLTM}.
\begin{proof}
We prove the following stronger statement: If $\lambda<\lambda_0$ then the equation $(\DeltaXinfty+\lambda)u=0$ can have no bounded solution, and if $\lambda=\lambda_0=0$ the only bounded solutions are constant on each $X_v^\infty$.

By Theorem \ref{mainthmquantumgraph}, with Remark 3 following it, this implies the theorem since $L^2$-solutions are bounded and since elements in the $(+1)$-eigenspace of $S(0)$ correspond to bounded solutions by Theorem \ref{thmscatt2}c).

First, by Lemma \ref{lemmasepvar}, for a bounded solution $u$ with $\lambda\leq\nu$ we must have
$\psi=0,\psi_k=0\,\forall k$ in \eqref{eqn1} resp.\ \eqref{eqn2a} and \eqref{eqn3}, and this implies
$$ \langle u_{x=0},(\partial_\xi u)_{x=0}\rangle_Y \leq 0.$$
The same is true for an $L^2$-solution for any $\lambda$.
Green's theorem implies
\begin{align}
\lambda\int_{X^0} |u|^2 &= \int_{X^0} u(-\Delta \ubar) = -\langle u_{x=0},(\partial_\xi u)_{x=0}\rangle_Y + \int_{X^0} |\nabla u|^2\\
 & \geq \int_{X^0} |\nabla u|^2,
\end{align}
so $\frac{\int_{X^0}|\nabla u|^2}{\int_{X^0}|u|^2}\leq\lambda$ if $u_{|X^0}\not\equiv0$. Since $u_{|X^0}$ may be taken as test function in the variational characterization of $\lambda_0$, this implies $\lambda_0\leq\lambda$ and hence the first claim. If $\lambda=\lambda_0=0$ then it implies $\nabla u\equiv 0$, so $u$ is constant on $X_v^0$ (since $X_v^0$ is connected) and hence on $X_v^\infty$ by unique continuation. Note that for $\nu=0$ connectedness of $Y_e$ implies that $\calN_e\cong\C$ canonically (the constant functions).
\end{proof}

\section{Appendix: Monotone unitary families} \label{sec monotone unitary}
In this appendix we collect some results on analytic one-parameter families of unitary operators which we need. Discussion and proofs can be found in \cite{Gri:MUF}.

Let $U(x)$ be a family of unitary operators on a Hermitian vector space $V$, of dimension $M<\infty$, depending real analytically on $x\in\R$. Then
\begin{equation}
\label{eqn monotone}
D(x) := \frac 1 i U'(x) U(x)^{-1}
\end{equation}
is symmetric, where $U'(x)$ is the derivative with respect to $x$.  Assume that $U$ is monotone, i.e. $D(x)$ is positive for all $x$, and more precisely that there are constants $\dmin,\dmax,d_2>0$  such that
\begin{equation}
\label{eqn dminmax def}
\dmin I\leq D(x) \leq \dmax I, \quad \|U''(x)\| \leq d_2\quad\text{ for all }x.
\end{equation}
Denote
$$ W(x)=\Ker (I-U(x))\quad\text{ and } \calZ = \{x:\, W(x)\neq \{0\}\}.$$
Thus $x\in\calZ$ iff $U(x)$ has eigenvalue one.

A special case of this setup is $U(x)=e^{ix}U_0$ for a unitary $U_0$. Then $\calZ$ is discrete and $2\pi$-periodic, and $W(x)$ is the eigenspace of $U_0$ with eigenvalue $e^{-ix}$.
The following statements generalize this and well-known facts about eigenspaces to our more general situation.

\begin{lemma}
\label{lemma counting ev}
$\calZ\subset\R$ is a discrete subset, and more precisely for all $A<B$
\begin{equation}
\label{eqn asymp}
\left| \sum\limits_{x:A< x < B} \dim W(x) - \frac 1 {2\pi}\int_A^B \tr D(x)\, dx \right| < M (:=\dim V)
\end{equation}
\end{lemma}
The following lemma mimics the independence of the eigenspaces.
\begin{lemma}
\label{lemma direct sum W}
Let $I$ be an interval of length at most $\frac {2\dmin}{d_2 M}$. Then the spaces $W(x),$ $x\in I$, are independent, i.e.
\begin{equation}
\label{eqn Wx independent}
\text{If }\phi_x\in W(x)\text{ for each }x\in I\cap\calZ\text{ and }\sum_x \phi_x=0\text{ then }\phi_x=0\ \forall x.
\end{equation}
\end{lemma}
The following lemma gives a stable version of almost orthogonality.
\begin{lemma} \label{lemma stability}
Assume $\phi\in V\setminus 0$ satisfies
\begin{equation}
\label{eqn phiest}
\|(I-U(x_0))\phi\| \leq \epsilon \|\phi\|.
\end{equation}
Then
\begin{equation}
\label{eqn dist to Z}
\dist (x_0,\calZ) \leq \frac {2\epsilon}{\dmin}.
\end{equation}
Furthermore, there is a constant $C$ only depending on $\dmin,\dmax,d_2,M$ such that if ${\epsilon}< C^{-1}$ then, with $P_W$ denoting the projection to $\bigoplus_{|x-x_0|\leq\sqrt\epsilon} W(x)$,
\begin{equation}
\label{eqn dist to sum of W}
\|\phi - P_W\phi\| \leq C\epsilon^{1/2(M+1)}\|\phi\|.
\end{equation}
\end{lemma}

We also need a fact about 2-parameter perturbations.
\begin{theorem}
\label{thm unitary perturb}
Let $U(x,y)$ be a unitary operator in a finite-dimensional Hermitian vector space depending real analytically on $x,y\in\R$. Assume
\begin{equation}
\label{eqn positivity}
\frac 1i \frac{\partial U}{\partial x} U^{-1} > 0\quad \text{ at } (x_0,y_0).
\end{equation}
Then the set $\{(x,y):\, U(x,y)\text{ has eigenvalue one}\}$ is, in a neighborhood of $(x_0,y_0)$, a union of real analytic curves $x=x_j(y)$. The corresponding projections $P_j(y)$ to the eigenspace of $U(x_j(y),y)$ with eigenvalue one are also analytic functions of $y\neq y_0$, extending analytically to $y=y_0$, and $\sum_j P_j(y_0)$ is the projection to $\ker (I-U(x_0,y_0))$.
\end{theorem}

%\bibliography{all}
%\bibliographystyle{amsplain}
\bibliographystyle{amsplain}
\bibliography{mypapers,dglib,mathlib}

\end{document}